\documentclass{article}%
\usepackage{eurosym}
\usepackage{amsmath}
\usepackage{amssymb}
\usepackage{amsfonts}
\usepackage{geometry}
\usepackage{graphicx}
\usepackage{float}
\usepackage{caption}
\usepackage{cite}
\usepackage{subcaption}
\usepackage{verbatim}
\usepackage[utf8]{inputenc}
\usepackage{placeins}

\providecommand{\U}[1]{\protect\rule{.1in}{.1in}}
\newtheorem{theorem}{Theorem}
\newtheorem{acknowledgement}[theorem]{Acknowledgement}

\newtheorem{corollary}[theorem]{Corollary}

\newtheorem{definition}[theorem]{Definition}
\newtheorem{example}[theorem]{Example}

\newtheorem{lemma}[theorem]{Lemma}

\newtheorem{remark}[theorem]{Remark}

\newenvironment{proof}[1][Proof]{\noindent\textbf{#1.} }{\ \rule{0.5em}{0.5em}}
\begin{document}

\title{Local Geometry of Self-similar Sets: Typical Balls, Tangent Measures
and Asymptotic Spectra}
\author{Manuel Mor\'{a}n $^{1,2}$, Marta LLorente$^3$ and Mar\'{\i}a Eugenia
Mera$^1$ }
\date{}
\maketitle

{\centering{\small {$^{1}$ Departamento de An\'{a}lisis Econ\'{o}mico y
Econom\'{\i}a Cuantitativa. Universidad Complutense de Madrid. Campus de
Somosaguas, 28223 Madrid, Spain.}}}\newline
{\centering{{\small {$^{2}$ IMI-Institute of Interdisciplinary Mathematics.
Universidad Complutense de Madrid. Plaza de Ciencias 3, 28040 Madrid, Spain.}%
}}}\newline
{\centering{\small {$^{3}$ Departamento de An\'{a}lisis Econ\'{o}mico: Econom%
\'{\i}a Cuantitativa. Universidad Aut\'{o}noma de Madrid, Campus de
Cantoblanco, 28049 Madrid, Spain.}}}

{\centering{Emails: mmoranca@ucm.es, m.llorente@uam.es, mera@ucm.es}}

\textit{Short Title:} Local Geometry of Self-similar Sets

\begin{abstract}
We analyse the local geometric structure of self-similar sets with open set
condition through the study of the properties of a distinguished family of
spherical neighbourhoods, the typical balls. We quantify the complexity of
the local geometry of self-similar sets, showing that there are uncountably
many classes of spherical neighbourhoods that are not equivalent under
similitudes. We show that, at a tangent level, the uniformity of the
Euclidean space is recuperated in the sense that any typical ball is a
tangent measure of the measure $\nu $ at $\nu $-a.e. point, where $\nu $ is
any self-similar measure. We characterise the spectrum of asymptotic
densities of metric measures in terms of the packing and centred Hausdorff
measures. As an example, we compute the spectrum of asymptotic densities of
the Sierpinski gasket.
\end{abstract}

\textit{[2020] MSC: 28A78, 28A80, 28A75, 54A05, 54A25}

\newpage
\section{Introduction and main results\label{Section Introduction}}

In order to gauge the vastness of the set of spherical neighbourhoods of a
metric space $X,$ it is useful to consider the quotient spaces $%
Sph_{X}/\simeq _{\mathcal{F}},$ where $Sph_{X}$ is the set of spherical
neighbourhoods of $X$ and $\simeq _{\mathcal{F}}$ is the equivalence class
associated with some group $\mathcal{F}$ of self-mappings of $X:$ $B\simeq _{%
\mathcal{F}}B^{\prime }\Leftrightarrow $ $B=f(B^{\prime })$ for $B,$ $%
B^{\prime }\in $ $Sph_{X}$ and some $f\in \mathcal{F}.$ The regularity of
the Euclidean space $\mathbb{R}^{n}$ is made clear by the fact that if $%
\mathcal{S}_{n}$ is the set of similarities of $\mathbb{R}^{n},$ then $Sph_{%
\mathbb{R}^{n}}/\simeq _{\mathcal{S}_{n}}$ consists of a unique equivalence
class.

In this paper, we study the local geometry of a self-similar set $E\subset 
\mathbb{R}^{n}$ satisfying the open set condition (OSC), geometry which is
described by the spherical neighbourhoods of $E$ as a metric subspace of $%
\mathbb{R}^{n},$ i.e. by restricted balls of the form $B\cap E,$ where $B$
is a Euclidean ball. For general points $x,$ $y\in E,$ if $B(x,d)$ denotes
the closed Euclidean ball centred at $x$ and with radius $d,$ then $%
B(x,d)\cap E$ and $B(y,d)\cap E$ are not equivalent by translation, and $%
B(x,d)\cap E$ and $B(x,d^{\prime })\cap E$ with $d\neq d^{\prime }$ are not
homothetic-equivalent. Using classical tools of fractal geometry, namely,
the $s$-densities of metric measures on balls (see Definitions \ref{Restrictedballs} and \ref{Spherical density}), and Marstrand's Theorem \cite{MARST}, together with the results in Sec. \ref{Asymptotic spectra}, we are
able to prove that, for general self-similar sets with OSC, there are
uncountably many equivalence classes in the quotient spaces $Sph_{E}/\simeq
_{\mathcal{S}_{n}}.$ This gives account of the complexity of the purely
deterministic self-similar geometry.

In spite of these facts, the literature has established the existence of a
strong kind of regularity, on a tangent level and on average, in the
neighbourhoods of a self-similar set.

Recall that a \textit{self-similar set} is defined as the unique compact set 
$E\subset \mathbb{R}^{n}$ that satisfies the basic equation of
self-similarity 
\begin{equation}
E=\cup _{i=0}^{m-1}f_{i}(E).  \label{1}
\end{equation}
for \ a given system $\Psi =\left\{ f_{i}\right\} _{i\in M},\ M:=\left\{
0,1,\dots ,m-1\right\} $ of contractive similitudes in $\mathbb{R}^{n}$. We
shall assume that the system $\Psi $ satisfies the OSC, meaning that there
is an open set $\mathcal{O\subset }\mathbb{R}^{n}$ such that $f_{i}\mathcal{%
(O)\subset O}$ for all $i\in M$ and $f_{i}(\mathcal{O)\cap }f_{j}(\mathcal{%
O)=\varnothing }$ for $i,$ $j\in M,\ i\neq j.$ We shall refer to such a set $%
\mathcal{O}$ as a \textit{feasible open set} for $\Psi .$ We can assume,
without loss of generality, as we shall from now on, that $\mathcal{\ O\cap }%
E\neq \varnothing $ holds, also called \textit{strong open set condition}
(SOSC) (cf. \cite{La} and \cite{Sc}, see also \cite{Mo2}). If $f_{i}(E)\cap
f_{j}(E)=\varnothing $ for $i,$ $j\in M,$ $i\neq j,$ it is said that the 
\textit{strong separation condition} (SSC) holds, in which case the OSC is
also fulfilled.

We want to understand the local geometry of $E$ through the study of the
local behaviour of the metric $s$-measures, 
\begin{equation}
\mathcal{M}^{s}\lfloor _{E}:=\left\{ \mu ,\text{ }\mathcal{H}^{s}\lfloor
_{E},\text{ }\mathcal{H}_{Sph}^{s}\lfloor _{E},\text{ }C^{s}\lfloor _{E},%
\text{ }P^{s}\lfloor _{E}\right\}  \label{restrictedmetric}
\end{equation}%
where $s$ is the \textit{similarity dimension} of $E,$ $\dim E,$ that is,
the unique real number $s$ that satisfies $\sum_{i\in M}r_{i}^{s}=1,$ $r_{i}$
being the contraction constant of the similarity $f_{i},\ i\in M.$ Here $%
\beta \lfloor _{E}$ stands for a measure $\beta $ restricted to the set $E.$
The measures 
\begin{equation}
\mathcal{M}^{s}:=\left\{ \mathcal{H}^{s},\text{ }\mathcal{H}_{Sph}^{s},\text{
}C^{s},\text{ }P^{s}\right\}  \label{MS}
\end{equation}%
are the $s$-dimensional Hausdorff measure, spherical Hausdorff measure,
centred Hausdorff measure and packing measure, respectively. Any two
measures in $\mathcal{M}^{s}\lfloor _{E}$ are multiple of each other,
moreover, \ in the case that $s$ takes the integer value $n,$ they are also
multiple of the $n$-dimensional Lebesgue measure. Each measure in $\mathcal{M%
}^{s}\lfloor _{E}$ highlights different basic geometric properties of
subsets of $\mathbb{R}^{n}.$ For $\alpha \in \mathcal{M}^{s}\lfloor _{E},$ $%
0<\alpha (E)<\infty $ holds and $E$ is called an $s$-set (see \cite{MAT2}
for further details and Sec. \ref{sectionmetricmeasures} for the definitions
of the measures in $\mathcal{M}^{s}$). We shall present in Sec. \ref%
{subsection self-similar measures} below the \textit{natural probability
measure }$\mu $. For the time being, we can see it as the normalised
measure, $\frac{\alpha }{\alpha (E)}$ of any other $\alpha \in \mathcal{M}%
^{s}\lfloor _{E}.$

The results in this paper about the regularity of the metric measures are
also shared by the wider class of self-similar measures, $\mathcal{M}_{%
\mathcal{S}}(E)$ (see \cite{HUTCH} and Sec. \ref{subsection self-similar
measures} for a definition). Whereas the metric measures, $\mathcal{M}^{s},$
convey a strong geometric meaning, self-similar measures are an essential
tool in multifractal analysis of logarithmic densities, a topic that has
generated a vast amount of literature for the past 30 years.

\subsection{Scenery flow, tangent distribution and tangent measures\label%
{Flow scenery}}

Let $\nu $ be a Radon measure on $\mathbb{R}^{n}$ and let $x$ be a point in
the support of $\nu .$ We can access the local geometry of $\nu \lfloor _{E}$
around $x$ through the following zooming process: let $T_{x,t}(y)=t(y-x),\
t>0,$ be the homothety that maps the ball $B(x,t^{-1})$ onto the unit ball $%
D:=B(0,1).$ Let $\nu _{x,t}$ be the probability measure on $D$ obtained from
the normalisation\ of the restriction to $D$ of the image measure of $\nu
\lfloor _{E}$ under the homothety $T_{x,t}.$ If $\mathcal{M(}D)$ denotes the
set of Radon measures on $D,$ then the mapping $t\rightarrow \nu _{x,t}$ can
be considered as a measure-valued time series that takes values in the
metric space $\mathcal{M(}D)$ endowed with the weak topology. This time
series is called \textit{scenery flow }of $\nu $ around $x$ (cf. \cite%
{Bedford}). The empirical distributions $\Phi _{x,t}(\nu ),$ $t>0,$
associated with such \textquotedblleft time\textquotedblright\ series, are
probability measures on $\mathcal{M(}D)$ (so they belong to the set $%
\mathcal{M}(\mathcal{M(}D))$ of Radon measures on $\mathcal{M(}D)).$ The
empirical distribution $\Phi _{x,t}(\nu )$ gives weight to a set $A\subset 
\mathcal{M(}D)$ according to the rate of the \textit{time} interval $[0,t]$
that the \textquotedblleft empirical\textquotedblright\ data $\delta _{\nu
_{x,t}}$ (unit mass at $\nu _{x,t})$ stay in $A.$ If the empirical
distribution $\Phi _{x,t}(\nu )$ converges to a limit $\Phi _{x}(\nu )$ as $t
$ tends to infinity, then the limiting distribution $\Phi _{x}(\nu )$ is
called the \textit{tangent distribution} of $\nu $ at $x$ (see \cite{BAN1}).

S. Graf \cite{Graf} proved that if $E$ is a self-similar set with OSC and $%
\nu \in \mathcal{M}_{\mathcal{S}}(E),$ then the limit $\Phi _{x}(\nu )$
exists $\nu $-a.e. $x,$ and it does not depend on $x.$ Moreover, he
constructed an explicit formula for the tangent distribution. This author
gave credit for the first of these results to C. Bandt in \cite{BAN1}, and
Bandt in turn gives credit for the same result to S. Graf \cite{BAN} (indeed
a most refreshing case). M. Arbeiter \cite{ARB}, C. Bandt \cite{BAN} and A.
Py\"{o}r\"{a}l\"{a} \cite{Pyorala} extended these results in different ways.
The uniqueness and independence of the limit $\Phi _{x}(\nu )$ from $x$ is
what M. Gavish, \cite{Gavish}, calls, when displayed by a measure, the 
\textit{uniform scaling scenery property }of such a measure. This means
that, at a tangent level and in this sense, the flow scenery recovers the
uniformity of the Euclidean space.

\begin{remark}
\label{tangent Preiss}There is another way to pass to the limit at the
tangent level that leads to \emph{tangent measures,} a concept prior to
tangent distributions introduced by D. Preiss \cite{Preiss}. There, starting
from a measure $\nu $ in the set $\mathcal{M(}\mathbb{R}^{n})$ of Radon
measures on $\mathbb{R}^{n},$ he considers unrestricted zoomings $\nu _{x,t}$
of $\nu $ at $x$ by homotheties $T_{x,t}$ as above. Instead of performing an
averaging procedure, Preiss considers non-null and locally finite limits, in
the vague topology of $\mathcal{M(}\mathbb{R}^{n}),$ of sequences%
\begin{equation*}
\left\{ c_{n}\nu _{x,t_{n}}\right\} \quad \text{with}\quad t_{n}\overset{n\rightarrow \infty }{\rightarrow }\infty \quad \text{and }\quad c_{n}>0.
\end{equation*}%
Such limit points are called \textit{tangent measures of }$\nu $ at $x,$ and 
$Tan(\nu ,x)$ denotes the set of all such limits. \newline
In our approach, following C. Bandt \cite{BAN1}, the measures $\nu
_{x,t_{n}} $ are restricted and normalised zoomings, but the zoomings are
through general expanding similitudes, rather than only homotheties.
\end{remark}

Let $\mathcal{I}_{n}$ be the group of isometries of $\mathbb{R}^{n}.$ We may
define, in the set $\mathcal{M}(\mathbb{R}^{n}),$ the equivalence
relationship 
\begin{equation}
\alpha \cong \beta \Leftrightarrow \text{there is a }g\in \mathcal{I}_{n}%
\text{ and a }\lambda >0\text{ such that }\beta =\lambda \left( g_{\sharp
}(\alpha \right) ),  \label{equivalent measures}
\end{equation}%
where $g_{\sharp }(\alpha )$ is the \textit{image measure} of $\alpha $
under $g,$ i.e. $g_{\sharp }(\alpha )(A)=\alpha (g^{-1}(A))$ for $\alpha $%
-measurable $A\subset \mathbb{R}^{n}.$ Thus, we identify two measures if
they are equal up to an isometry (see, for instance, \cite{BAN}, where
equivalent measures up to isometries are identified in the construction of
tangent measures), and we also identify all measures in the half-straight
line $\left\{ \lambda \alpha :\lambda >0,\text{ }\alpha \in \mathcal{M}(%
\mathbb{R}^{n})\right\} .$ For $\alpha \in \mathcal{M(}\mathbb{R}^{n}),\,$
let $\widetilde{\alpha }$ denote the equivalence class in $\mathcal{M(}%
\mathbb{R}^{n})/\cong $ to which $\alpha $ belongs, i.e. 
\begin{equation}
\widetilde{\alpha }=\left\{ \beta \in \mathcal{M(}\mathbb{R}^{n}):\beta
\cong \alpha \right\}  \label{7}
\end{equation}%
Given a measure $\nu \in \mathcal{M(}\mathbb{R}^{n}),$ we now consider the
zoomings $\nu _{x,t_{n}}$ be of the form $\left( g_{n}\right) _{\sharp }$ $%
\nu \lfloor _{B(x,dt_{n}^{-1})}$ where $g_{n}$ is a similitude of
contraction ratio $t_{n},$ $d\leq 1,$ and $x\in spt(\nu )$ (see (\ref{10})).
We define the quotient space $\widetilde{\mathcal{M}}(\mathbb{R}^{n})$ and
the set of \textit{tangent equivalence classes of measures, }$\widetilde{Tan}%
(\nu ,x),$ by
\begin{eqnarray}
\widetilde{\mathcal{M}}(\mathbb{R}^{n}) &=&\left\{ \widetilde{\alpha }%
:\alpha \in \mathcal{M(}\mathbb{R}^{n})\right\}   \label{8} \\
\widetilde{Tan}(\nu ,x) &=&\left\{ \widetilde{\alpha }:\text{there is a
sequence }c_{n}\nu _{x,t_{n}}\xrightarrow[n \to \infty]{w}\alpha ,\text{with 
}t_{n}\rightarrow \infty ,\text{ }\alpha \neq 0\text{ and }\alpha \in 
\mathcal{M(}\mathbb{R}^{n})\right\} ,  \label{9}
\end{eqnarray}%
where $\overset{w}{\rightarrow }$ denotes the weak convergence of measures
on $\mathcal{M}(\mathbb{R}^{n}).$ 

It turns out that, in the course of our
research, the case in which the convergence of the magnifications occurs in
the strong topology of measures in $\mathcal{M}(\mathbb{R}^{n})$ is relevant
(see Sec. \ref{Typical balls} below for a discussion of this result). We
shall write $\widetilde{Tan}^{st}(\nu ,x)$ for the set of equivalence
classes, w.r.t. $\cong ,$ of such strong limits.

\begin{remark}
In our definition (\ref{9}) any two zoomings, $\beta =\left( g_{n}\right)
_{\sharp }$ $\nu \lfloor _{B(x,dt_{n}^{-1})}$ and $\beta ^{\prime }=\left(
h_{n}\right) _{\sharp }\nu \lfloor _{B(x,dt_{n}^{-1})}$ of a given spherical
neighbourhood $B(x,dt_{n}^{-1})$ are considered as valid steps in the
construction of a tangent limiting measure $\alpha ,$ where $g_{n},h_{n}$
are different similitudes. This can be considered as the identification of $%
\beta $ and $\beta ^{\prime }$ as equivalent zoomings. Notice that $\beta
^{\prime }=\left( g_{n}^{-1}\circ h_{n}\right) _{\sharp }\beta $ and that $%
g_{n}^{-1}\circ h_{n}$ is an isometry. Thus, the equivalence relationship (%
\ref{equivalent measures}) and the definition in (\ref{9}) are consistent.
\end{remark}

In contrast to the enlightening results obtained in \cite{Graf}, \cite{BAN}
and \cite{ARB} on the uniform scaling scenery property of self-similar
measures, to the best of our knowledge, the members of $Tan(\nu ,x)$ for $%
\nu \in \mathcal{M}_{\mathcal{S}}(E)$ remain unknown$.$ Several natural
issues arise here: What is the relationship between $\Phi _{x}(\nu )$ and $%
Tan(\nu ,x)?$ What do the measures in $Tan(\nu ,x)$ look like? Do they
display some uniform property? As for the first question, see Proposition 1
in \cite{Mortens}. Below, we give a partial answer to the second and third
questions for measures in $\mathcal{M}_{\mathcal{S}}(E)$ (see (\ref%
{uniformtangentmeasures}) and Theorem \ref{quasitangent}).

\subsection{Typical balls \label{Typical balls}}

A distinguished class of neighbourhoods of $E$, in terms of which our
results are expressed, is the class of \textit{typical balls}.

\begin{definition}
\label{typicalball} A ball $B(x,d)$ is said to be typical if $x\in E$ and $%
B(x,d)\subset\mathcal{O},$ where $\mathcal{O}$ is some feasible open set. We
shall write $\mathcal{B}$ for the set of typical balls.
\end{definition}

The family of typical balls is invariant under the semigroup $G$ generated
by $\Psi $ (see Sec. \ref{Notation}), since, for $f\in G,$ it follows from $%
f(\mathcal{O)\subset O}$ that $f\mathcal{(B)\subset B}$ holds. Consider now
the set of \textit{typical spherical }$\mathcal{B}$-\textit{measures,} 
\begin{equation}
\mathcal{M}_{\mathcal{S}}\mathcal{(B)}:=\left\{ \alpha \lfloor B:B\in 
\mathcal{B},\ \alpha \in \mathcal{M}_{S}(E)\right\} .
\label{typical b-measures}
\end{equation}%
It is well known \cite{HUTCH} that, for any $x\in E,$ the set $\left\{
f(x):f\in G\right\} $ is dense in $E,$ so the balls in $\mathcal{B}$ are
typical in the sense that, if $B\in \mathcal{B},$ then similar copies of $B$
are densely spread over $E$ at small scales by the action of $G.$ These
copies are a countable set of balls. As Theorem \ref{quasitangent} shows,
the measures in $\mathcal{M}_{\mathcal{S}}\mathcal{(B)}$ are also typical in
a deeper sense since, for any $f\in G,$ $B\in \mathcal{B}$ and $\alpha \in 
\mathcal{M}_{\mathcal{S}}\mathcal{(B)},$ the equality $\alpha \lfloor
_{f(B)}=p_{f}f_{_{\sharp }}(\alpha \lfloor _{B})$ holds for a certain
constant $p_{f}<1$ associated with $f.$ This means that the images of
typical balls are identical copies, up to the constant $p_{f},$ to the
original ones not only as subsets, but also from the point of view of any
property expressible in terms of self-similar measures. Moreover, in Theorem %
\ref{quasitangent} it is shown that, for any typical ball $B(x,d),$ for any
measure $\alpha \in \mathcal{M}_{\mathcal{S}}\mathcal{(}E\mathcal{)}$ and
for all points $y$ in a set $\widehat{E}$ with full $\alpha $-measure, there
is a sequence of balls $\left\{ B(y,d_{k})\right\} $ with $d_{k}\rightarrow
0,$ a sequence $\left\{ f_{k}\right\} $ of similitudes in $G$ and constants $%
p_{f_{k}}^{-1}\rightarrow \infty ,$ such that 
\begin{equation}
p_{f_{k}}^{-1}\left( f_{k}^{-1}\right) _{\sharp }\left( \alpha \lfloor
_{(B(y,d_{k})}\right) \xrightarrow[k\to \infty]{st}\alpha \lfloor _{B(x,d)},
\label{reescaled}
\end{equation}%
where the convergence in (\ref{reescaled}) is in the sense of the strong
topology of Radon measures.

Theorem \ref{quasitangent} also states that, for all $x\in \widehat{E}$ and $%
\alpha \in \mathcal{M}_{\mathcal{S}}\mathcal{(B)},$ 
\begin{equation}
\widetilde{\mathcal{M}}_{\mathcal{S}}\mathcal{(B)}\subset \widetilde{Tan}%
^{st}(\alpha ,x)  \label{uniformtangentmeasures}
\end{equation}%
holds, where 
\begin{equation}
\widetilde{\mathcal{M}}_{\mathcal{S}}\mathcal{(B)=}\left\{ \widetilde{\alpha 
}:\alpha \in \mathcal{M}_{\mathcal{S}}\mathcal{(B)}\right\} ,
\label{typical measures}
\end{equation}%
(see (\ref{7}) for the notation $\widetilde{\alpha }).$

The results above imply that the use of general zooming similitudes, grants
the strong convergence of the zoomings to the tangent measures, whereas in
the ordinary spaces of tangent measures, where only homotheties are allowed,
convergence can only be ensured in a weak topology sense. See Sec. \ref%
{Identifications of tangent measures} below for further details on
identifications and topologies of measures.

\begin{remark}
Putting the results in Sec. \ref{complexity}, described in the first
paragraph of this section, together with (\ref{uniformtangentmeasures}), we
see that the self-similar scenery at $x\in E$ depends on $x$ on large
scales, meaning that there is a broad variety of balls $B(x,d)$ for varying $%
x$ that, moreover, also vary with $d$ for fixed $x.$ Additionally, on a
tangent scale, for each $\alpha \in \mathcal{M}_{\mathcal{S}}\mathcal{(B)}$
and each $x\in \widehat{E},$ each typical class of measures in $\widetilde{%
\mathcal{M}}_{\mathcal{S}}\mathcal{(B)}$ is a feasible outcome of the
zooming process of $\alpha $ at $x,$ so there is a wide variety of limiting
measures in $\widetilde{Tan}^{st}(\alpha ,x),$ $x\in \widehat{E}.$ The
uniformity of the self-similar setting emerges here in the fact that the
inclusion $\widetilde{\mathcal{M}}_{\mathcal{S}}\mathcal{(B)}\subset 
\widetilde{Tan}^{st}(\alpha ,x)$ stands true for any $x\in \widehat{E},$ so
all the points in $\widehat{E}$ share the set $\widetilde{\mathcal{M}}_{%
\mathcal{S}}\mathcal{(B)}$ of tangent measures.
\end{remark}

\subsection{Spectrum of local densities of a self-similar set: the
Sierpinski gasket case\label{spectrumintrodution}}

The relevance of the typical balls is stressed by the connection between
typical balls and the spectrum of densities, which in turn determines some
basic geometric features of $E.$

Let $\alpha \in \mathcal{M(}\mathbb{R}^{n}),$ $0\leq s\leq \infty $ and $%
x\in \mathbb{R}^{n}.$ The \textit{upper and lower spherical }$s$-densities
of $\alpha $ at $x$ are defined, respectively, by%
\begin{align}
\overline{\theta }_{\alpha }^{s}(x)& =\limsup_{d\rightarrow 0}\text{ }\theta
_{\alpha }^{s}(x,d),  \label{upper-density} \\
\underline{\theta }_{\alpha }^{s}(x)& =\liminf_{d\rightarrow 0}\text{ }%
\theta _{\alpha }^{s}(x,d),  \label{lower-density}
\end{align}%
where the $s$-density of the ball $B(x,d),$ $\theta _{\alpha }^{s}(x,d),$ is
given by 
\begin{equation*}
\theta _{\alpha }^{s}(x,d)=\frac{\alpha (B(x,d))}{(2d)^{s}}.
\end{equation*}%
\qquad

Here the zooming process is summarised in only two scalars, $%
\eqref{upper-density}$ and $\eqref{lower-density}$. If $\overline{\theta }%
_{\alpha }^{s}(x)=\underline{\theta }_{\alpha }^{s}(x),$ then we write $%
\theta _{\alpha }^{s}(x)$ for the common value and call it $s$-density of $%
\alpha $ at $x$. Densities and their connections to their underlying
measures have been studied extensively in the context of geometric measure
theory. A major contribution from Marstrand (Marstrand's theorem, \cite%
{MARST}) asserts that, in the Euclidean setting, if the $s$-density $\theta
_{\alpha }^{s}(x)$ exists in a set with a finite and positive $\alpha $%
-measure with $\alpha \in \mathcal{M(}\mathbb{R}^{n}),$ then $s$ is an
integer.

The widest class of subsets of Euclidean spaces that are $s$-sets (i.e. sets
with a finite and positive $\alpha $-measure) is either the class of
self-similar sets that satisfy the OSC, with $s$ being their \textit{%
similarity dimension }(see \eqref{Notation}), or some variations of it, like
the Mauldin and Williams graph-directed constructions, cf. \cite{MW}, and
controlled Moran constructions, cf. \cite{MoP}. Here, we are interested in
the case in which the similarity dimension $s$ is not an integer and, by
Marstrand's theorem described above, $\underline{\theta }_{\alpha }^{s}(x)$
and $\overline{\theta }_{\alpha }^{s}(x)$ do not coincide in subsets with a
positive $\alpha $-measure. This leads to the following definition of 
\textit{asymptotic spectrum} of densities of a given measure $\alpha $ at a
point and, more in general, in a subset of points.

\begin{definition}
Given a subset $A\subset\mathbb{R}^{n},$ we define the \emph{asymptotic
spectrum of (non-logarithmic) spherical $s$-densities,} $Spec(\alpha,A),$
for a locally finite measure $\alpha$ by 
\begin{equation}
Spec(\alpha,A)=\left\{
\lim_{k\rightarrow\infty}\theta_{\alpha}^{s}(x,d_{k}):x\in A\text{ and }%
\lim_{k\rightarrow\infty}d_{k}=0\right\} .  \label{Spec}
\end{equation}
\end{definition}

We insert the non-logarithmic epithet above because there is a ample
literature on the so-called \textit{multifractal spectrum }of logarithmic
spherical densities. This literature also focuses on the limiting behaviour
of $\alpha $ on small balls, but the interest is in the upper and lower
limits of the quotients $\frac{\log \alpha (B(x,d))}{\log d}$ when $%
d\rightarrow 0$ (for $x\in E)$ and, in particular, in the fractal dimension
of both the ($\alpha $-null) sets where these limits exist and take
particular values \cite{HARTE} and the sets of \textit{divergence points}
(see \cite{CAJAR}, \cite{COLEBROOK}, \cite{LI}) where the limits do not
coincide. Much less is known about the behaviour of non-logarithmic
densities, and the research in this paper can be considered a preliminary
step in that direction.

In particular, in Sec. \ref{section spectrum}, Theorem~\ref{Spectrum}, we
present the knowledge to date about the spectrum of non-logarithmic $\alpha $%
-densities, $\alpha \in \mathcal{M}^{s}\lfloor _{E},$ of self-similar sets $%
E $ that satisfy the OSC. In particular, we show that $Spec(\alpha ,x)$ is
contained in the closed interval $\left[ \frac{\alpha (E)}{P^{s}(E)},\frac{%
\alpha (E)}{C^{s}(E)}\right] $ for all $x$ in a subset $\widehat{E}$ of $E$
with a full $\alpha $-measure. There arises a natural class of self-similar
sets with nice properties, the $\alpha $-\textit{exact self-similar sets}
(see notation in ~\ref{exact}), which are sets for which the endpoints of
such interval belong to $Spec(\alpha ,x),$ $x\in \widehat{E}.$ Whereas the
results for general self-similar sets with OSC presented in Sec. \ref%
{section spectrum} are of a qualitative nature, in Sec. \ref%
{Sierpinski_gasket} we shall focus on our prime example of $\alpha $-exact
self-similar set, the Sierpinski gasket $S,$ and exploit its regularity to
accurately approximate the range of values taken by its spectrum, which is
the content of Theorem~\ref{Sierpinskispectrum}. Moreover, we give a full
characterisation of the spectrum of all the points in $S,$ which is given by
the union of two closed intervals of positive length, namely, 
\begin{equation*}
Spec(\alpha ,S)=\left[ \alpha (S)\underline{\theta }_{\mu
}^{s}(z_{0}),\alpha (S)\overline{\theta }_{\mu }^{s}(z_{0})\right] \cup %
\left[ \frac{\alpha (S)}{P^{s}(S)},\frac{\alpha (S)}{C^{s}(S)}\right] ,\text{
}\alpha \in \mathcal{M}^{s}\lfloor _{S},
\end{equation*}%
where $z_{0}:=(0,0).$ Using the numerical approximations of $\underline{%
\theta }_{\mu }^{s}(z_{0}),$ $\overline{\theta }_{\mu }^{s}(z_{0})$ obtained
in Sec. \ref{Sierpinski_gasket} and of $P^{s}(S)$ and $C^{s}(S)$ obtained in 
\cite{LLMM2} and \cite{LLMM3}, we can also show that these two intervals are
disjointed. In the case that $\alpha \in \{\mu ,P^{^{s}}\lfloor
_{S},C^{s}\lfloor _{S}\},$ we have numerical estimations of these two
disjointed intervals. The Sierpinski gasket is, as far as we know, the first
connected self-similar with non-integer dimension for which the entire
spectrum has been computed.

\section{Notation and preliminaries\label{Notation}}

The self-similar set $E$ given in (\ref{1}) can be parametrised as $%
E=\left\{ \pi (i):i\in \Sigma \right\} $ with parameter space $\Sigma
:=M^{\infty }$ and \textit{geometric projection mapping} $\pi :\Sigma
\rightarrow E$ given by $\pi (i)=\cap _{k=1}^{\infty }f_{i(k)}E,\ $where $%
i(k)$ denotes the curtailment $i_{1}\dots i_{k}\in M^{k}$ of $%
i=i_{1}i_{2}\dots \in \Sigma $ and $f_{i_{1}\dots i_{k}}=f_{i_{1}}\circ
f_{i_{2}}\circ f_{i_{3}}\circ ...f_{i_{k}}$. We adopt the convention $%
M^{0}=\varnothing $ and write $M^{\ast }=\cup _{k=0}^{\infty }M^{k}$ for the
set of words of finite length. Expressed in this notation, the semigroup
generated by $\Psi $ can be written as $G=\left\{ f_{i}:i\in M^{\ast
}\right\} .$

For any $i\in M^{\ast },$ we denote by $E_{i}$ the cylinder sets $f_{i}(E),$
and if $i\in M^{0},$ then $f_{i}(E):=E.$ The sets $E_{i}$ are called $k$%
-cylinders if $i\in M^{k}.$ We also shorten the notation $f_{i}(A)$ to $%
A_{i} $ for a general set $A\subset \mathbb{R}^{n}.$ We write $%
r_{i}:=r_{i_{1}}r_{i_{2}}\dots r_{i_{k}}$ for the contraction ratio of the
similitude $f_{i}.$

Moreover, $\sigma \ :\Sigma \rightarrow \Sigma \ $ shall stand for the 
\textit{shift} map given by $\sigma (i_{1}i_{2}i_{3}\dots
)=i_{2}i_{3}i_{4}\dots $ The code shift can be projected (as a
correspondence) onto $E,$ yielding the \textit{geometric shift}%
\begin{equation}
\mathcal{T}(x):=\pi \circ \sigma \circ \pi ^{-1}(x),  \label{shift}
\end{equation}%
$x\in E.$ The \textit{shift orbit} of $x\in E$ is given by $\left\{ \mathcal{%
T}^{k}(x):k\in \mathbb{N}\right\} .$

\begin{remark}
\label{Geometric shift} Observe that $x\in\mathcal{T}^{k}(A)$ if and only if 
$f_{i}(x)\in A$ for some $i\in M^{k}.$
\end{remark}

\subsection{Self-similar measures\label{subsection self-similar measures}}

Let $\mathcal{P}(\mathbb{R}^{n})$ be the space of compactly supported
probability Borel measures on $\mathbb{R}^{n},$ let $\mathbf{p}%
=(p_{0},...,p_{m-1})\in\mathbb{R}^{m}$ be a probability vector and let $%
\mathbf{M}_{\mathbf{p}}\mathbf{:}$ $\mathcal{P}(\mathbb{R}^{n})\rightarrow 
\mathcal{P}(\mathbb{R}^{n})$ be the Markov operator defined by

\begin{equation*}
\mathbf{M}_{\mathbf{p}}(\alpha )=\sum_{i=0}^{m-1}p_{i}\alpha \circ
f_{i}^{-1},\text{ }\alpha \in \mathcal{P}(\mathbb{R}^{n}).
\end{equation*}%
The unique fixed point of the contractive operator $\mathbf{M}_{\mathbf{p}}$
is called the \textit{self-similar measure }$\mu _{\mathbf{p}};$ that is, 
\begin{equation}
\mu _{\mathbf{p}}=\sum_{i\in M}p_{i}\mu _{\mathbf{p}}\circ f_{i}^{-1}.
\label{invariant3}
\end{equation}%
Moreover, 
\begin{equation}
\mathbf{M}_{\mathbf{p}}^{k}(\alpha )=\sum_{i\in M^{k}}p_{i}\alpha \circ
f_{i}^{-1}\underset{k\rightarrow \infty }{\xrightarrow{w}}\mu _{\mathbf{p}}
\label{markov}
\end{equation}%
for any $\alpha \in \mathcal{P}(\mathbb{R}^{n}),$ where, for $i\in M^{k},$ $%
p_{i}:=$ $p_{i_{1}}\cdot \cdot \cdot p_{i_{k}}.$ Here $\mathbf{M}_{\mathbf{p}%
}^{k}$ is the $k$-th iterate of $\mathbf{M}_{\mathbf{p}}$ (see \cite{HUTCH}
and \cite{BARNS} for further details). Set 
\begin{equation}
\mathcal{M}_{\mathcal{S}}(E):=\left\{ \mu _{\mathbf{p}}:\text{ }%
\sum_{i=0}^{m-1}p_{i}=1,\ p_{i}>0,\text{ }i=0,...,m-1\right\} .
\label{self-similar-measure}
\end{equation}%
For $\mathbf{p}_{s}:=(r_{0}^{s},...,r_{m-1}^{s}),$ where $s$ is the
similarity dimension of $E$ (recall that $r_{i}$ is the contraction constant
of the similarity $f_{i},$ $i\in M),$ the measure $\mu _{\mathbf{p}_{s}}$ is
called the\textit{\ natural probability measure} on $E.$ Furthermore, if $%
\alpha \in \mathcal{M}^{s}\lfloor _{E}$ (see \eqref{restrictedmetric} for
notation), then%
\begin{equation}
\mu :=\mu _{\mathbf{p}_{s}}=\frac{\alpha }{\alpha (E)}  \label{muvshaus}
\end{equation}%
(see \cite{LLM3}).

Notice that, whereas the measures in $\mathcal{M}^{s}$ (see (\ref{MS}) for
notation) convey an strong geometrical meaning, the measures $\mu _{\mathbf{p%
}}$ in $\mathcal{M}_{\mathcal{S}}(E)$ do not. They are concentrated in dense
subsets of $E,$ $E_{\mathbf{p}},$ whose dimension is given by $\dim (E_{%
\mathbf{p}})=$ $s_{\mathbf{p}}:=\frac{\sum_{i=0}^{m-1}p_{i}\log p_{i}}{%
\sum_{i=0}^{m-1}p_{i}\log r_{i}},$ but the measure $\mu _{\mathbf{p}}$ is
singular w.r.t. the measures $\mathcal{H}^{s_{\mathbf{p}}}$ and $P^{s_{%
\mathbf{p}}}$ (see \cite{Mo3} and \cite{Mo4}).

\subsection{Metric measures\label{sectionmetricmeasures}}

We now briefly recall metric measures. They are the classical tools for
analysing the geometric properties of subsets of $\mathbb{R}^{n}.$

The \textit{Hausdorff centred measure}, $C^{s}(A),$ of a subset $A\subset 
\mathbb{R}^{n},$ was defined by Saint Raymond and Tricot \cite{SRT} in a
two-step process. First, the premeasure $C_{0}^{s}(A)$ is defined for any $%
s>0$ by 
\begin{equation}
C_{0}^{s}(A)=\lim_{\delta \rightarrow 0}\inf \left\{
\sum\limits_{i=1}^{\infty }(2d_{i})^{s}\ :\ 2d_{i}\leq \delta ,\text{ }%
i=1,2,\dots \right\} ,  \label{premeasure}
\end{equation}%
where the infimum is taken over all coverings, $\left\{
B(x_{i},d_{i})\right\} _{i\in \mathbb{N}^{+}},$ of $A$ by closed balls $%
B(x_{i},d_{i})$ centred at points $x_{i}\in A.$ Then, the \textit{centred
Hausdorff }$s$\textit{-dimensional measure }is defined by 
\begin{equation*}
C^{s}(A)=\sup \left\{ C_{0}^{s}(F):F\subset A,\ F\text{ closed}\right\} .
\end{equation*}

The second step in the definition of $C^{s}(A)$ is due to the lack of
monotonicity of $C_{0}^{s}$ (see \cite{TO2} and \cite[Example~4]{LLM2}).
However, in \cite{LLM2}, it was shown that the second step can be omitted
when restricting oneself to self-similar sets with OSC.

With regard to metric measures based on packings, the standard packing
measure $P^{s}$ (see \cite{SRT} and \cite{SU}) is also defined in a two-step
process, 
\begin{equation*}
P_{0}^{s}(A)=\lim_{\delta \rightarrow 0}\sup \left\{
\sum\limits_{i=1}^{\infty }(2d_{i})^{s}:2d_{i}\leq \delta ,\text{ }%
i=1,2,\dots \right\} ,
\end{equation*}%
where the supremum is taken over all \textit{packings} $\left\{
B(x_{i},d_{i})\right\} _{i\in \mathbb{N}^{+}},$ with $x_{i}\in A$ for all $%
i, $ and $B(x_{i},d_{i})\cap B(x_{j},d_{j})=\varnothing $ for $i\neq j.$
Then, 
\begin{equation*}
P^{s}(A)=\inf \left\{ \sum\limits_{i=1}^{\infty }P_{0}^{s}(F_{i})\right\} ,
\end{equation*}%
where the infimum is taken over all coverings $\left\{ F_{i}\right\} _{i\in 
\mathbb{N}^{+}}$ of $A$ by closed sets $F_{i}$ (cf.\cite{T0}). In \cite{FENG}%
, it was proved that if $A$ is a compact set with $P_{0}^{s}(A)<\infty ,$
then $P^{s}(A)=P_{0}^{s}(A),$ so this simplification applies to any compact
subset of a self-similar set with OSC.

The \textit{spherical }$s$\textit{-dimensional Hausdorff measure, }$\mathcal{%
H}_{Sph}^{s}(A),$ is obtained by removing in (\ref{premeasure}) the
requirement that the balls are centred at points of $A$. The classical $s$%
\textit{-dimensional Hausdorff measure, }$\mathcal{H}^{s}(A),$ results if
coverings of $A$ by arbitrary subsets, $\left\{ U_{i}\right\} _{i\in \mathbb{%
\ N}^{+}},$ are considered and $2d_{i}$ is replaced in (\ref{premeasure})
with the diameter of $U_{i},$ $\left\vert U_{i}\right\vert $ (see \cite{Haus}
and \cite{MAT2}). No second step is required for these last two measures.

The packing and the centred Hausdorff measures have a much simpler
expression when dealing with self-similar sets $E$ that satisfy the OSC as
the browse for optimal packings or coverings can be reduced to the search
for optimal density balls within the class of typical balls, $\mathcal{B}$
(see Definition~\ref{typicalball}). In particular, for any self-similar $E$
that satisfies the OSC and with similarity dimension $s,$ it is known (see 
\cite{Mo1}) that 
\begin{equation}
P^{s}(E)=\left( \inf \left\{ \theta _{\mu }^{s}(x,d):B(x,d)\in \mathcal{B}%
\right\} \right) ^{-1},  \label{packing}
\end{equation}%
and, Lemma (\ref{centeredtypballs}) of Sec. \ref{Asymptotic spectra} implies
that 
\begin{equation}
C^{s}(E)=\left( \sup \left\{ \theta _{\mu }^{s}(x,d):B(x,d)\in \mathcal{B}%
\right\} \right) ^{-1}.  \label{centrada}
\end{equation}

\section{Local structure and typical balls\label{section spectrum}}

Now we shall study the local structure of a self-similar set $E$ that
satisfies the OSC for a feasible open set $\mathcal{O}$ through the study of
the scenery flow of $\alpha \in \mathcal{M}_{\mathcal{S}}(E)$ at a.e. $x\in
E,$ and through the characterisation of the spectrum of the spherical $s$%
-densities of measures in $\mathcal{M}^{s}\lfloor _{E}$ (Sec. \ref%
{Asymptotic spectra}), a limiting set that helps to summarise the structure
in the neighbourhood of a point (Sec. \ref{spectrumintrodution}).

\subsection{Scenery flow and tangent measures\label{Flow scenery and tangent
measures}}

We start by giving details on the construction of $\widetilde{Tan}(\nu ,x)$
for $\nu \in \mathcal{M}_{\mathcal{S}}(E)$ and $x\in E$ (see \ref%
{self-similar-measure} for notation).

\subsubsection{Tangent measures, identifications and topologies.\label%
{Identifications of tangent measures}}

Recall that the construction of the sets $\widetilde{Tan}(\nu ,x)$ and $%
\widetilde{Tan}^{st}(\nu ,x)$ employs the identification, in the set $%
\mathcal{M}(\mathbb{R}^{n})$, of those measures that are equal up to
isometries or mutual multiples (see (\ref{equivalent measures}), (\ref{7}), (%
\ref{8}) and (\ref{9}) for notation). We now examine the construction of the
spaces of equivalence classes of tangent measures above in more detail.

For $\nu \in \mathcal{M}(\mathbb{R}^{n})$ and $x\in spt(\nu ),$ we first
consider$\ $sequences $\{c_{n}\nu _{x,t_{n}}\}_{n=0}^{\infty },$ where for
every $n\in \mathbb{N},$ $c_{n}>0,$ 
\begin{equation}
\nu _{x,t_{n}}:=\frac{1}{\nu (B(x,dt_{n}^{-1}))}\left( g_{x,t_{n}}\right)
_{\sharp }\nu \lfloor _{B(x,dt_{n}^{-1})},  \label{10}
\end{equation}%
$d\leq 1$ and $g_{x,t_{n}}$ is some similarity with expanding ratio $t_{n}$
that maps the ball $B(x,t_{n}^{-1})$ onto the ball $B(z_{n},1),$ with $%
z_{n}=g_{x,t_{n}}(x),$ so each $\nu _{x,t_{n}}$ is a probability measure
supported on $B(z_{n},d).$ Then, $\widetilde{Tan}(\nu ,x)$ and $\widetilde{%
Tan}^{st}(\nu ,x)$ consist of the equivalence classes of non-null weak and
strong limits, respectively, as $t_{n}\rightarrow \infty ,$ of such
sequences $\{c_{n}\nu _{x,t_{n}}\}_{n=0}^{\infty }$ (see (\ref{9})). Lemma~%
\ref{classes} shows that the elements in $\widetilde{Tan}(\nu ,x)$ and $%
\widetilde{Tan}^{st}(\nu ,x)$ do not depend on either the sequence of
constants $c_{n}$ or the particular elements chosen in the equivalence
classes $\widetilde{\nu _{x,t_{n}}}$ as long as the convergence of these
elements is guaranteed.

\begin{remark}
The unit ball $D$ does not play any essential role in our definition of
tangent measures in the quotient space $\widetilde{\mathcal{M}}(\mathbb{R}%
^{n}).$ In the opposite direction (second approach) we may, in a more akin
way to the classical approach, require the similarities $g_{x,t_{n}}$ to map 
$B(x,t_{n}^{-1})$ onto $B(0,1),$ and then define $Tan_{D}(\nu ,x)$ and $%
Tan_{D}^{st}(\nu ,x)$ as weak and strong limits in $\mathcal{M}(D),$
respectively, of sequences of such measures $\nu _{x,t_{n}},$ and $%
\widetilde{Tan}(\nu ,x),$ $\widetilde{Tan}^{st}(\nu ,x)$ as the sets of
equivalence classes of measures in $Tan_{D}(\nu ,x)$ and $Tan_{D}^{st}(\nu
,x),$ respectively. \newline
This second method gives spaces of tangent equivalence classes which are
particular cases of these in our primary approach. Are these equivalent
methods? In order to answer this question, let a sequence $\{c_{n}\nu
_{x,t_{n}}\}_{n=0}^{\infty },$ $\ $as in (\ref{10}), converge to a non-null
Radon measure $\alpha .$ By Lemma \ref{classes} we may assume $c_{n}=1$ for
all $n\in 
\mathbb{N}
^{+}.$ Since the measures $\nu _{x,t_{n}}$ are supported on balls $%
B(z_{n},d) $ with $d\leq 1$ (see Theorem \ref{quasitangent} (i)), the
limiting measure $\alpha $ must also be supported on a ball $B(z,d)$ with $%
z_{n}\underset{n\rightarrow \infty }{\longrightarrow }z.$ Each measure $\nu
_{x,t_{n}}^{\prime }=\left( \tau _{z_{n}}\right) _{\sharp }\nu _{x,t_{n}},$
\ where $\tau _{z_{n}}(y)=y-z_{n},$ is equivalent by translation to $\nu
_{x,t_{n}},$ and $\nu _{x,t_{n}}^{\prime }$ is supported on $D.$ It is easy
to see that $\nu _{x,t_{n}}\xrightarrow[n \to \infty]{w}\alpha $ implies
that $\nu _{x,t_{n}}^{\prime }\xrightarrow[n \to \infty]{w}\alpha ^{\prime
}=\left( \tau _{z}\right) _{\sharp }\alpha ,$ so $\alpha ^{\prime }$ is
equivalent to $\alpha $ and supported on $D.$ Thus, the second method gives
the same space $\widetilde{Tan}(\nu ,x)$ than our primary method. But $\nu
_{x,t_{n}}\xrightarrow[n \to \infty]{st}\alpha $ does not imply that $\nu
_{x,t_{n}}^{\prime }\xrightarrow[n \to \infty]{st}\alpha ^{\prime },$ so the
second method does not produce the same space $\widetilde{Tan}^{st}(\nu ,x)$
than our method. \newline
Observe that, if we let $\nu _{x,t_{n}}^{\prime }=\left( \tau _{z}\right)
_{\sharp }\nu _{x,t_{n}},$ then $\nu _{x,t_{n}}\xrightarrow[n \to \infty]{st}%
\alpha $ does imply $\nu _{x,t_{n}}^{\prime }\xrightarrow[n \to \infty]{st}%
\alpha ^{\prime }=\left( \tau _{z}\right) _{\sharp }\alpha .$ But now the
measure $\nu _{x,t_{n}}^{\prime }$ is supported on the ball $B(z_{n}-z,d)$
rather than on $D.$ This observation is useful because $D$ and all the balls 
$B(z_{n}-z,d)$ are contained in some ball $B(0,R)$ for $R$ large enough
(notice that $z_{n}$ is a convergent sequence of points), so the convergence 
$\nu _{x,t_{n}}\underset{n\rightarrow \infty }{\longrightarrow }\alpha $ (weak or strong) occurs in $\mathcal{M(}B(0,R)),$
and we can see that, if we consider vague convergence of measures, we do not
obtain anything new, since in the Polish space $B(0,R)$ both convergences
are equivalent (\cite{Mortens}, Appendix).
\end{remark}

\begin{lemma}
\label{classes} \hfill \break (i) The sequences $\{c_{n}\}_{n=0}^{\infty }$
in the construction of $\ \widetilde{Tan}(\nu ,x)$ and $\widetilde{Tan}%
^{st}(\nu ,x)$ can be taken to be $c_{n}=1,$ $n=0,1,2,...$ \newline
(ii) Let $\nu \in \mathcal{M}(\mathbb{R}^{n}),$ $x\in spt(\nu )$ and $\alpha
\in Tan(\nu ,x).$ Let $\{t_{n}\}_{n=0}^{\infty }\uparrow \infty $ be such
that $\{\nu _{x,t_{n}}\}_{n=0}^{\infty }$ $\xrightarrow[n \to \infty]{w}%
\alpha .$ Assume also that there is a sequence $\{f_{n}\}_{n=0}^{\infty }$ $%
\ $in the set of isometries $\mathcal{I}_{n}$ such that $\{\left(
f_{n}\right) _{\sharp }\nu _{x,t_{n}}\}_{n=0}^{\infty }%
\xrightarrow[n \to
\infty]{w}\alpha ^{\prime }.$ Then, there is $f$ $\in \mathcal{I}_{n}$ such
that $(f)_{\sharp }\alpha =\alpha ^{\prime }.$ The same is true if $\ $the
convergence holds in the topology of the strong convergence in $\mathcal{M}(%
\mathbb{R}^{n}).$
\end{lemma}

\begin{proof}
\hfill \break (i) By definition, a weak limiting measure $\alpha $ as in (%
\ref{9}) is a non-null measure in $\mathcal{M}(\mathbb{R}^{n}).$ Therefore,
the sequence of constants $\left\{ c_{n}\right\} $ must be bounded above and
below by two positive and finite constants. We can choose a subsequence $%
\left\{ c_{n_{k}}\right\} _{k=0}^{\infty }$ that converges to a constant $c,$
and then the sequence $c\nu _{x,t_{n}}$ must converge to the weak
limit $\alpha .$ This gives $\nu _{x,t_{n}}\xrightarrow[n \to \infty]{w}%
c^{-1}\alpha ,$ which belongs to the same equivalence class in $\widetilde{%
Tan}(\nu ,x)$ as $\alpha .$ On the other hand, the non-null weak limits in $%
\mathcal{M}(\mathbb{R}^{n})$ of sequences $\{\nu _{x,t_{n}}\}_{n=0}^{\infty
} $ are particular cases of those of sequences $\left\{ c_{n}\nu
_{x,t_{n}}\right\} .$ This completes the proof of part (i) for weak limits$.$
The argument also holds true for strong limits. \newline
(ii) For any $n\in \mathbb{N}^{+},$ we can write $f_{n}(\cdot )=g_{n}(\cdot
)+a_{n},$ where $g_{n}$ is an orthogonal map and $a_{n}\in \mathbb{R}^{n}.$
Recall that $\nu _{x,t_{n}}$ is supported on $B(z_{n},d),$ so $\left(
g_{n}+a_{n}\right) _{\sharp }(\nu _{x,t_{n}})$ is supported on $%
a_{n}+B(z_{n},d)$ (\cite{MAT2}, Theorem 1.18), with $z_{n}\underset{%
n\rightarrow \infty }{\longrightarrow }z.$ This means that, if $\nu
_{x,t_{n}}^{\prime }:=\left( f_{n}\right) _{\sharp }\nu _{x,t_{n}}$
converges, in the weak topology of $\mathcal{M}(\mathbb{R}^{n}),$ to some
non-null measure $\alpha ^{\prime }$ in $\mathcal{M}(\mathbb{R}^{n}),$ the
sequence $a_{n}$ must be bounded, and then the sequence $\{f_{n}\}_{n=0}^{%
\infty }$ is also bounded in the supremum norm. Therefore, there is a
convergent subsequence, $\{f_{n_{k}}\}_{k=0}^{\infty },$ of $%
\{f_{n}\}_{n=0}^{\infty }.$ Let $f:=\lim_{k\rightarrow \infty }f_{n_{k}}.$
Since the sequence $\{\left( f_{n_{k}}\right) _{\sharp }(\nu
_{x,t_{n_{k}}})\}_{k=0}^{\infty }$ converges to $\alpha ^{\prime },$ we have
that 
\begin{equation}
\alpha ^{\prime }=\lim_{k\rightarrow \infty }f_{n_{k}\sharp }(\nu
_{x,t_{n_{k}}})=f_{\sharp }\alpha ,  \label{limit}
\end{equation}%
which proves that $\alpha ^{\prime }\cong \alpha .$ The second equality in (%
\ref{limit}) holds true because, for any $\varphi $ in the space $C_{0}(%
\mathbb{R}^{n})$ of continuous, compactly supported functions on $\mathbb{R}%
^{n}$ and for any $\varepsilon >0,$ there is $k_{0}>0$ such that for $k\geq
k_{0},$ we have 
\begin{equation*}
\left\Vert \varphi \circ f_{n_{k}}-\varphi \circ f\right\Vert \leq \frac{%
\varepsilon }{2},
\end{equation*}%
\begin{equation*}
\left\Vert \int \varphi \circ f\text{ }d(\nu _{x,t_{n_{k}}})-\int \varphi
\circ f\text{ }d\alpha \right\Vert \leq \frac{\varepsilon }{2},
\end{equation*}%
and then 
\begin{align*}
& \left\Vert \int \varphi \text{ }d\left( f_{n_{k}}{}_{\sharp }(\nu
_{x,t_{n_{k}}})\right) -\int \varphi \text{ }d(f_{\sharp }\alpha )\right\Vert
\\
& \leq \int \left\Vert \varphi \circ f_{n_{k}}-\varphi \circ f\text{ }%
\right\Vert d(\nu _{x,t_{n_{k}}})+\left\Vert \int \varphi \circ f\text{ }%
d(\nu _{x,t_{n_{k}}})-\int \varphi \circ f\text{ }d\alpha \right\Vert \leq
\varepsilon .
\end{align*}%
If $\{\nu _{x,t_{n}}^{\prime }\}_{n=0}^{\infty }$ converges to $\alpha
^{\prime }$ in the strong topology of $\mathcal{M}(D),$ then it also
converges in the weak topology and the argument above applies.
\end{proof}

\subsubsection{Scaling properties of typical balls and scenery flow}

We need some preliminary lemma and the following definition.

\begin{definition}
Given a measure $\alpha\in\mathcal{M}^{s}\lfloor_{E},$ two Euclidean balls $%
B(x,d)$ and $B(x^{\prime},d^{\prime})$ are said to be $\alpha$-density
equivalent if $\theta_{\alpha}^{s}(x,d)=\theta_{\alpha}^{s}(x^{\prime
},d^{\prime}).$
\end{definition}

We start with two elementary scaling properties of typical balls for
measures in $\mathcal{M}_{\mathcal{S}}(E)$ and in $\mathcal{M}%
^{s}\lfloor_{E}.$

\begin{lemma}
\label{scalingselfsimilar}Let $E$ be a self-similar set generated by the
system $\ \Psi =\left\{ f_{i}\right\} _{i\in M}$ of similarities of $\ 
\mathbb{R}^{n},$ with  $M=\left\{ 0,1,\dots ,m-1\right\}, $ and similarity
dimension $s.$ Let $\mathcal{O}$ be a feasible open set (for $\Psi )$ and
let $i\in M^{\ast }.$ Then \newline
(i) 
\begin{equation}
\mu _{\mathbf{p}}(f_{i}(A))=p_{i}\mu _{\mathbf{p}}(A),\text{ for }\mu _{%
\mathbf{p}}\in \mathcal{M}_{\mathcal{S}}(E)\quad \text{and }\mu _{\mathbf{p}}%
\text{-measurable }A\subset \mathcal{O},  \label{selfsimilarscaling}
\end{equation}%
(ii) 
\begin{equation}
\mu _{\mathbf{p}}(f_{i}^{-1}(C))=p_{i}^{-1}\mu _{\mathbf{p}}(C)\text{ for }%
\mu _{\mathbf{p}}\in \mathcal{M}_{\mathcal{S}}(E)\text{ and }\mu _{\mathbf{p}%
}\text{-measurable }C\subset \mathcal{O}_{i},
\label{inversescalingselfsimilar}
\end{equation}%
(iii)%
\begin{equation}
B(f_{i}(x),r_{i}d)\text{ is}\ \alpha \text{-density equivalent to}\ B(x,d)%
\text{ for }\alpha \in \mathcal{M}^{s}\lfloor _{E}\ \text{and }B(x,d)\subset 
\mathcal{O},  \label{directscaling}
\end{equation}%
(iv)%
\begin{equation}
f_{i}^{-1}(B(f_{i}(x),r_{i}d))\text{ is}\ \alpha \text{-density equivalent to%
}\ B(x,d)\ \text{for }\alpha \ \in \mathcal{M}^{s}\lfloor _{E}\ \text{and }\
B(x,d)\subset \mathcal{O}_{i}.  \label{inversescaling}
\end{equation}
\end{lemma}

\begin{proof}
The proof of \ (\ref{selfsimilarscaling}) is trivial from (\ref{invariant3})
if $E$ satisfies SSC. If SSC does not hold, then%
\begin{equation*}
\mu _{\mathbf{p}}(f_{j}^{-1}(f_{i}(A)))\leq \mu _{\mathbf{p}}(\partial 
\mathcal{O)}=0\text{ for }j\neq i,
\end{equation*}%
because $A\subset \mathcal{O}$ and, hence, $f_{j}^{-1}(f_{i}(A))\cap
E\subset \partial \mathcal{O},$ which is known to be a $\mu _{\mathbf{p}}$-$%
\,$null set (cf. \cite{Mo4}), so (\ref{selfsimilarscaling}) also follows
from (\ref{invariant3}). If we set $A=f_{i}^{-1}(C)$ in (\ref%
{selfsimilarscaling}), we obtain (\ref{inversescalingselfsimilar}) (see also 
\cite{BAN}). By (\ref{muvshaus}), we can apply (\ref{selfsimilarscaling})
and (\ref{inversescalingselfsimilar}) to any measure $\alpha \in \mathcal{M}%
^{s}\lfloor _{E},$ which easily gives (\ref{directscaling}) and (\ref%
{inversescaling}).
\end{proof}

Before stating the main theorem of this section, we will see the following
lemma.

\begin{lemma}
\label{imagemeasureslemma} \hfill \break (i) Let $g,$ $f:\mathbb{R}%
^{n}\rightarrow \mathbb{R}^{n},$ $\alpha \in \mathcal{M}(\mathbb{R}^{n}),$ $%
\lambda >0$ and $A\subset \mathbb{R}^{n}$ be an $\alpha $-measurable subset.
Then, the following equalities hold true:\newline
$\mathbf{\bullet }$ $\lambda \left( g\right) _{\sharp }(\alpha )=\left(
g\right) _{\sharp }(\lambda \alpha ),\newline
\mathbf{\bullet }$ $(f\circ g)_{\sharp }\alpha =f_{\sharp }(g)_{\sharp
}(\alpha ),$ and\newline
$\mathbf{\bullet }$ $\left( g_{\sharp }\alpha \right) \lfloor _{A}=g_{\sharp
}(\alpha \lfloor _{g^{-1}(A)})$\newline
(ii) Let $\alpha $ be a measure on $\mathcal{M}(\mathbb{R}^{n}),$ $g:\mathbb{%
R}^{n}\rightarrow \mathbb{R}^{n}$ a bijective mapping and $\beta :=g_{\sharp
}\left( \alpha \right) .$ Then, $\alpha =(g^{-1})_{\sharp }\beta .$\newline
(iii) If $\left\{ \alpha _{k}\right\} _{k\in \mathbb{N}}$ is a sequence of
measures on $\mathcal{M}(\mathbb{R}^{n})$ and $(g)_{\sharp }\left( \alpha
_{k}\right) \xrightarrow[k \to \infty]{st}\beta ,$ then $\alpha _{k}%
\xrightarrow[k \to
\infty]{st}\left( g^{-1}\right) _{\sharp }\left( \beta \right) .$\newline
(iv) Let $B(x_{n},d):=B_{n}$ be a sequence of closed balls that converges in
the Hausdorff metric to a closed ball $B(x,d):=B,$ and let $\alpha \in 
\mathcal{M}(B)$ with $\alpha (\partial B)=0.$ Then $\alpha \lfloor
_{B_{n}}:=\alpha _{n}$ $\xrightarrow[n \to \infty]{st}\alpha .$
\end{lemma}

\begin{proof}
\newline
Parts (i)-(iii) easily follows from the definitions. \newline
Recall that $\alpha _{n}\xrightarrow[n \to \infty]{st}$ $\alpha $ means that 
$\alpha _{n}(A)\underset{n\rightarrow \infty }{\longrightarrow }\alpha (A)$
for any Borel set $A\subset \mathbb{R}^{n}.$ Let $\alpha \in \mathcal{M(}B)$
and let $K$ be any compact set contained in the interior $U$ of $B.$ The
distance $d(K,\partial B)=\min \left\{ \left\Vert x-y\right\Vert :x\in K,%
\text{ }y\in \partial B\right\} $ must be a quantity $\varepsilon >0$ and
then $K\subset B(x,d-\varepsilon ).$ The convergence of $B_{n}$ to $B$
implies that there is an $n_{0}\in \mathbb{N}^{+}$ such that, for $n>n_{0},$ 
$\left\Vert x-x_{n}\right\Vert \leq \varepsilon .$ Then, if $z\in $ $K,$%
\begin{equation*}
\left\Vert z-x_{n}\right\Vert \leq \left\Vert z-x\right\Vert +\left\Vert
x-x_{n}\right\Vert \leq d,
\end{equation*}%
which shows that $K\subset B\cap B_{n}$ for $n>n_{0}.$ Then, for such values
of $n,$ we have%
\begin{equation*}
\alpha _{n}\left( K\right) =\alpha (B_{n}\cap K)=\alpha (K)
\end{equation*}%
We now prove that $\alpha _{n}\xrightarrow[n \to \infty]{st}\alpha $ also
holds in the $\sigma $-field $\mathfrak{B}(B)$ of Borel subsets of $B.$ Let 
\begin{equation*}
\mathcal{A}:=\left\{ A\subset B:A\text{ is }\alpha \text{-measurable and }%
\lim_{n\rightarrow \infty }\alpha _{n}(A)=\alpha (A)\right\} .
\end{equation*}%
(Notice that any $\alpha $-measurable set is also $\alpha _{n}$-measurable
for all $n\in 
\mathbb{N}
^{+}).$ It is easy to check that $B\in \mathcal{A},$ that $B-A:=A^{c}\in 
\mathcal{A}$ if $A\in \mathcal{A},$ and that $\mathcal{A}$ is closed under a
finite union of its members or, in short, that $\mathcal{A}$ is a field. Let 
$F_{k}$ be a sequence of members of $\mathcal{A}.$ In order to show that $%
\cup _{k\in \mathbb{N}^{+}}F_{k}\in \mathcal{A},$ we first write $\cup
_{k\in \mathbb{N}^{+}}F_{k}=$ $\cup _{k\in \mathbb{N}^{+}}G_{k},$ where $%
G_{k}=\cup _{i=1}^{k}F_{i}.$ This shows that $\cup _{k\in \mathbb{N}%
^{+}}F_{k}$ can be expressed as a countable union of the increasing sequence 
$G_{k}$ of members of $\mathcal{A}.$ Furthermore, $\cup _{k\in \mathbb{N}%
^{+}}F_{k}=$ $\cup _{k\in \mathbb{N}^{+}}H_{k},$ where $%
H_{k}=(G_{k}-G_{k-1}) $ with $G_{0}=\varnothing .$ Now, each $H_{k}\in 
\mathcal{A}$ and $H_{k}\cap H_{k^{\prime }}=\varnothing $ for $k\neq
k^{\prime }.$ Then, using that each $\alpha _{n}$ is a measure, we have 
\begin{equation*}
\lim_{n\rightarrow \infty }\alpha _{n}\left( {\displaystyle%
\bigcup\limits_{k\in \mathbb{N}^{+}}}H_{k}\right) ={\displaystyle%
\sum\limits_{k\in \mathbb{N}^{+}}}\lim_{n\rightarrow \infty }\alpha
_{n}\left( H_{k}\right) ={\displaystyle\sum\limits_{k\in \mathbb{N}^{+}}}%
\alpha \left( H_{k}\right) =\alpha \left( {\displaystyle\bigcup\limits_{k\in 
\mathbb{N}^{+}}}H_{k}\right) .
\end{equation*}%
This completes the proof that $\mathcal{A}$ is a $\sigma $-field. Notice
that any closed set $K\subset B$ can be written as the union of the $\alpha $
and $\alpha _{n}$-null set $K\cap \partial B$ and of the set $K-\partial B,$
which belongs to $\mathcal{A}$ as a countable union of compact sets $K\cap
B(x,d-n^{-1})\subset U.$ Thus, the class $\mathcal{K}$ of closed subsets of $%
B$ is contained in $\mathcal{A}.$ We know that the $\sigma $-fields
generated by $\mathcal{K}$ and by $\mathcal{A}$ satisfy $\mathfrak{B}%
\mathcal{(}B)=\sigma (\mathcal{K)\subset \sigma (A)=A}.$ This gives the
strong convergence of $\alpha _{n}$ to $\alpha $ on $\mathfrak{B}\mathcal{(}%
B).$
\end{proof}

We can now go to the scenery flow of measures in $\mathcal{M}_{\mathcal{S}%
}(E).$

\begin{theorem}
\label{quasitangent}Let $E$ be a self-similar set generated by the system $%
\Psi =\left\{ f_{i}\right\} _{i\in M}$ of similarities on $\mathbb{R}^{n},$
with $M=\left\{ 0,1,\dots ,m-1\right\} $ and similarity dimension $s.$ Let $%
\mathcal{O}$ be a feasible open set (for $\Psi $) and $\mu _{\mathbf{p}}\in 
\mathcal{M}_{\mathcal{S}}(E).$ Then, for any $\mu _{\mathbf{p}}$-measurable
set $B\subset \mathcal{O}$ and $i\in M^{\ast },$ the following statements
hold true.\newline
(i) 
\begin{equation*}
\mu _{\mathbf{p}}\lfloor _{B_{i}}=p_{i}\left( f_{i}\right) _{\sharp }\left(
\mu _{\mathbf{p}}\lfloor _{B}\right) .
\end{equation*}%
(ii) 
\begin{equation}
\mu _{\mathbf{p}}\lfloor _{B}=p_{i}^{-1}\left( f_{i}^{-1}\right) _{\sharp
}\left( \mu _{\mathbf{p}}\lfloor _{B_{i}}\right) .  \label{1b}
\end{equation}%
(iii) There is a subset $\widehat{E}\subset E$ with $\mu _{\mathbf{p}}(%
\widehat{E})=1$ such that, if $x\in E$ and $B(x,d)\subset \mathcal{O},$ then
for any $y\in \widehat{E},$ there is a sequence $\left\{ i_{j}\right\}
_{j\in \mathbb{N}^{+}}$ with $i_{j}\in M^{\ast }$ and a sequence of balls $%
\left\{ B(y,dr_{i_{j}})\right\} _{j\in \mathbb{N}^{+}}$ such that 
\begin{equation*}
p_{i_{j}}^{-1}\left( f_{i_{j}}^{-1}\right) _{\sharp }\left( \mu _{\mathbf{p}%
}\lfloor _{B(y,dr_{i_{j}})}\right) \ \xrightarrow[j \to \infty]{st}\mu _{%
\mathbf{p}}\lfloor _{B(x,d)}
\end{equation*}%
(iv) For any $x\in \widehat{E},$ 
\begin{equation*}
\widetilde{\mathcal{M}}_{\mathcal{S}}\mathcal{(B)\subset }\widetilde{Tan}%
^{st}(\mu _{\mathbf{p}},x),
\end{equation*}%
where $\mathcal{M}_{\mathcal{S}}\mathcal{(B)}$ is defined in (\ref{typical
b-measures}).
\end{theorem}

\begin{proof}
In order to show (i), let $\mu _{\mathbf{p}}\in \mathcal{M}_{\mathcal{S}%
}(E), $ $i\in M^{\ast }$ and let $B\subset \mathcal{O}$ and $A\subset 
\mathbb{R}^{n}$ be $\mu _{\mathbf{p}}$-measurable sets. Then, 
\begin{align*}
\left( p_{i}({f_{i}})_{\sharp }\left( \mu _{\mathbf{p}}\lfloor _{B}\right)
\right) (A)& =p_{i}\left( \mu _{\mathbf{p}}\lfloor _{B}\right) \left(
f_{i}^{-1}(A\right) ) \\
& =p_{i}\mu _{\mathbf{p}}(f_{i}^{-1}(A\cap B_{i}))=\mu _{\mathbf{p}}\lfloor
_{B_{i}}(A),
\end{align*}%
where the third equality follows from \eqref{inversescalingselfsimilar} and
(i) is proved. Analogously, (ii) follows from \eqref{selfsimilarscaling}.%
\newline
Now, let 
\begin{equation}
\widehat{E}=\left\{ y\in E:\left\{ \mathcal{T}^{k}(y):k\in \mathbb{N^{+}}%
\right\} \text{ is dense in }E\right\}  \label{hatE}
\end{equation}%
(see (\ref{shift}) in Sec. \ref{Notation} the definition of $\mathcal{T}).$
It is well known (cf. \cite{Wal}) that the set $\widehat{E}$ has a full $\mu
_{\mathbf{p}}$-measure. Let $x\in E,$ $B(x,d)\subset \mathcal{O},$ $y\in 
\widehat{E}$ and $\left\{ x_{j}\right\} _{j\in 
\mathbb{N}
^{+}}$ such that $\lim_{j\rightarrow \infty }x_{j}=x$ (in the Euclidean
metric) with $x_{j}\in \mathcal{T}^{k_{j}}(y)$ for every ${j}\in \mathbb{%
N^{+}}.$ We may also assume that $B(x_{j},d)\subset \mathcal{O}$ for every $%
j\in \mathbb{N^{+}}.$ We shorten $B(x_{j},d)$ to $B^{j}$ and $B(x,d)$ to $B.$
Since $\lim_{j\rightarrow \infty }x_{j}=x,$ it follows that $\left\{
B^{j}\right\} _{j\in \mathbb{N}}$ converges to $B$ in the Hausdorff metric.
Also, $\mu _{\mathbf{p}}(\partial B)=0$ because $\mu _{\mathbf{p}}\in 
\mathcal{M}_{\mathcal{S}}(E)$ (cf. \cite{Mat1}). Then, Lemma~\ref{imagemeasureslemma} {(iv)} implies that 
\begin{equation}
\mu _{\mathbf{p}}\lfloor _{B^{j}}\xrightarrow[j \to \infty]{st}\mu _{\mathbf{%
p}}\lfloor _{B}.  \label{2}
\end{equation}%
Now, notice that, since $x_{j}\in \mathcal{T}^{k_{j}}(y)$ for each $j\in 
\mathbb{N},$ there is $i_{j}$ $\in M^{k_{j}}$ such that $f_{i_{j}}(x_{j})=y$
(see Remark \ref{Geometric shift}). Then, $%
f_{i_{j}}^{-1}(B(y,dr_{i_{j}}))=B^{j}.$ By \eqref{1b} applied to $B^{j}$ and 
$i_{j}\in M^{\ast },$ we see that 
\begin{equation}
\mu _{\mathbf{p}}\lfloor _{B^{j}}=p_{i_{j}}^{-1}\left( f_{i_{j}}^{-1}\right)
_{\sharp }\left( \mu _{\mathbf{p}}\lfloor _{{B(y,dr_{i_{j}})}}\right) ,
\label{3}
\end{equation}%
which concludes the proof of (iii).\newline
Observe that, in the terminology of Sec. \ref{Identifications of tangent
measures}, the right hand term in (\ref{3}) is, $c_{t_{j}}\nu _{y,t_{j}}$
for $\nu =\mu _{\mathbf{p}},$ $t_{j}=r_{i_{j}}^{-1}$ and $%
c_{t_{j}}=p_{i_{j}}^{-1}\mu _{\mathbf{p}}(B(y,dr_{i_{j}}))$ (recall that $%
\nu _{y,t_{j}}$ was a normalised blowup and notice also that we may assume,
rescaling $E$ if necessary, that all typical balls have a radius $d\leq 1).$
So, (\ref{2}) and (\ref{9}) give $\widetilde{\mu _{\mathbf{p}}\lfloor _{B}}%
\in \widetilde{Tan}^{st}(\mu _{\mathbf{p}},x)$ and part (iv) is proved.
\end{proof}

\subsection{Asymptotic spectra and measure-exact self-similar sets \label%
{Asymptotic spectra}}

We shall write $\mathit{Im}(\theta _{\alpha }^{s},\mathcal{B})$ to designate
the set 
\begin{equation*}
\mathit{Im}(\theta _{\alpha }^{s},\mathcal{B}):=\left\{ \theta _{\alpha
}^{s}(x,d):B(x,d)\in \mathcal{B}\right\}
\end{equation*}%
(see notation in Definition~\ref{typicalball}), which plays a relevant role
in the geometric analysis of $E$ (see (\ref{packing}) and the lemma below).

\begin{lemma}
\label{centeredtypballs} Let $E$ be a self-similar set generated by the
system of similarities of $\mathbb{R}^{n},\ \Psi =\left\{ f_{i}\right\}
_{i\in M},$ with $M=\left\{ 0,1,\dots ,m-1\right\} ,$ and similarity
dimension $s.$ If $E$ satisfies the OSC, then\newline
(i) 
\begin{equation*}
C^{s}(E)=\left( \sup \left\{ \theta _{\mu }^{s}(x,d):B(x,d)\in \mathcal{B}_{%
\mathcal{O}}\right\} \right) ^{-1},
\end{equation*}%
where $\mathcal{B}_{\mathcal{O}}:=\left\{ B(x,d)\in \mathcal{B}:\text{ }%
B(x,d)\subset \mathcal{O}\right\} $ and $\mathcal{O}$ is any feasible open
set for $\Psi .$\newline
(ii)%
\begin{equation*}
C^{s}(E)=\left( \sup \mathit{Im}(\theta _{\mu }^{s},\mathcal{B})\right)
^{-1}.
\end{equation*}
\end{lemma}

\begin{proof}
It is known that for a general self-similar set that satisfies the OSC\ (see 
\cite{Mo1} and \cite{LLM2}), 
\begin{equation}
C^{s}(E)=\left( \sup \{\theta _{\mu }^{s}(x,d):x\in E\text{ and }%
d>0\}\right) ^{-1}  \label{centradasup}
\end{equation}%
holds. Let $\mathcal{O}$ be any feasible open set. Then, it is enough to
show that 
\begin{equation*}
\sup_{(x,d)\in E\times \mathbb{R}^{+}}\theta _{\mu }^{s}(x,d)\leq
\sup_{B(x,d)\in \mathcal{B}_{\mathcal{O}}}\theta _{\mu }^{s}(x,d).
\end{equation*}%
Should this not be the case, there would exist $(x_{0},d_{0})\in E\times 
\mathbb{R}^{+}$ such that $B(x_{0},d_{0})\notin \mathcal{B}_{\mathcal{O}}$
and 
\begin{equation*}
\theta _{\mu }^{s}(x_{0},d_{0})>\sup_{B(x,d)\in \mathcal{B}_{\mathcal{O}%
}}\theta _{\mu }^{s}(x,d).
\end{equation*}%
In order to show that this contradicts (\ref{centradasup}), take $x^{\ast
}\in E\cap \mathcal{O}$ such that there is $i\in M^{\ast }$ with $%
f_{i}(x^{\ast })=x^{\ast }.$ Let $\rho _{1}:=\min \left\{ \left\Vert x^{\ast
}-z\right\Vert :z\in \partial \mathcal{O}\right\} .$ Observe that, if we
take $\rho _{2}>0$ so that $B(x_{0},d_{0})\subset B(x^{\ast },\rho _{2})$
and $k\in \mathbb{N}^{+},$ satisfying that $r_{i}^{k}\rho _{2}<\rho _{1},$
then 
\begin{equation*}
f_{i}^{k}(B(x_{0},d_{0}))\subset f_{i}^{k}(B(x^{\ast },\rho _{2}))=B(x^{\ast
},r_{i}^{k}\rho _{2})\subset \mathcal{O},
\end{equation*}%
which, using that $f_{i}^{k}(B(x_{0},d_{0})\cap S)\subset
f_{i}^{k}(B(x_{0},d_{0}))\cap S,$ raises the contradiction 
\begin{equation*}
\theta _{\mu }^{s}(x_{0},d_{0})\leq \frac{r_{i}^{-ks}\mu
(f_{i}^{k}(B(x_{0},d_{0})))}{(2d_{0})^{s}}=\frac{\mu
(B(f_{i}^{k}(x_{0}),r_{i}^{k}d_{0}))}{(2d_{0}r_{i}^{k})^{s}}\leq
\sup_{B(x,d)\in \mathcal{B}_{\mathcal{O}}}\theta _{\mu }^{s}(x,d).
\end{equation*}%
Part (ii) is trivial from (i).
\end{proof}

In the next theorem, we shall establish the relationships between the
pointwise and global spectra, the set $\mathit{Im}(\theta _{\alpha }^{s},%
\mathcal{B})$ and its extreme values $\alpha (E)\left( P^{s}(E)\right) ^{-1}$
and $\alpha (E)\left( C^{s}(E)\right) ^{-1}.$

\begin{theorem}
\label{Spectrum} Let $E\subset \mathbb{R}^{n}$ be a self-similar set that
satisfies the SOSC with feasible open set $\mathcal{O}$ and similarity
dimension $s,$ and let $\alpha \in \mathcal{M}^{s}\lfloor _{E}.$ Then, the
following statements hold true.\newline
(i) For $x\in E,$ it holds that 
\begin{equation*}
Spec(\alpha ,x)=\left[ \underline{\theta }_{\alpha }^{s}(x),\overline{\theta 
}_{\alpha }^{s}(x)\right]
\end{equation*}%
(see (\ref{lower-density}) and (\ref{upper-density}) for notation) 
\begin{equation*}
Spec(\alpha ,E)\subset \left[ \kappa _{1},\kappa _{2}\right]
\end{equation*}%
with $0<\kappa _{1}\leq \kappa _{2}<\infty .$\newline
(ii) There is a subset $\widehat{E}\subset E$ with $\mu (\widehat{E})=1$
such that, for any $y\in \widehat{E},$ 
\begin{equation*}
Spec(\alpha ,y)=Spec(\alpha ,\widehat{E})=Spec(\alpha ,\mathcal{O}\cap E).
\end{equation*}%
(iii) 
\begin{equation*}
\left( \frac{\alpha (E)}{P^{s}(E)},\frac{\alpha (E)}{C^{s}(E)}\right)
\subset \mathit{Im}(\theta _{\alpha }^{s},\mathcal{B})\subset Spec(\alpha ,%
\mathcal{O\cap }E)\subset \left[ \frac{\alpha (E)}{P^{s}(E)},\frac{\alpha (E)%
}{C^{s}(E)}\right] .
\end{equation*}
\end{theorem}

\begin{proof}
That $\underline{\theta }_{\alpha }^{s}(x)$ and $\overline{\theta }_{\alpha
}^{s}(x)$ belong to and are the extreme values of $Spec(\alpha ,x)$ follows
from the definitions. That all the intermediate values in between also
belong to $Spec(\alpha ,x)$ is a consequence of the continuousness of $%
\theta _{\alpha }^{s}(x,d),$ with respect to $d.$ This last property follows
from the fact that the $\alpha $-measure of the boundary of Euclidean balls
is always null \cite{Mat1} for any measure $\alpha \in \mathcal{M}_{\mathcal{%
S}}(E),$ which proves the first assertion of (i). The second assertion is
well known \cite{HUTCH}.\newline
In order to prove (ii), let $\widehat{E}$ be the full $\mu $-measure subset
of points of $E$ that have a dense geometric shift orbit in $E$ (see %
\eqref{hatE}) and let $y\in \widehat{E}.$ The inclusions $Spec(\alpha
,y)\subset Spec(\alpha ,\widehat{E})\subset Spec(\alpha ,\mathcal{O}\cap E)$
are trivial as $\widehat{E}\subset \mathcal{O}.$ This follows from the fact
that, if $y\notin \mathcal{O},$ then $\mathcal{T}(y)\cap \mathcal{O}%
=\varnothing $ because $f_{i}(\mathcal{O})\subset \mathcal{O}$ for any $i\in
M,$ and repeating the same argument, we see that $T^{k}(y)$ could not be
dense in $E.$\newline
The corresponding equalities would follow if we prove $Spec(\alpha ,\mathcal{%
O}\cap E)\subset Spec(\alpha ,y).$ This holds true because, if $%
z=\lim_{k\rightarrow \infty }\theta _{\alpha }^{s}(x,d_{k})$ for $x\in 
\mathcal{O\cap }E$ and $d_{k}\underset{k\rightarrow \infty }{\longrightarrow 
}0,$ since $B(x,d_{k})\in \mathcal{B}$ for any sufficiently large $k,$ we
can apply Theorem~\ref{quasitangent} {(iii)} to see that, for such values of 
$k,$ $\theta _{\alpha }^{s}(x,d_{k})\in Spec(\alpha ,y)$ and, hence, $%
\lim_{k\rightarrow \infty }\theta _{\alpha }^{s}(x,d_{k})\in Spec(\alpha ,y)$
easily follows from (i). This ends the proof of (ii). \newline
Finally, the first inclusion in (iii) for $\alpha =\mu $ follows from the
continuousness of the function $\theta _{\mu }^{s}(x,d)$ on $\mathbb{R}%
^{n}\times \mathbb{R}^{+}$ since 
\begin{equation*}
\frac{1}{P^{s}(E)}\leq \theta _{\mu }^{s}(x,d)\leq \frac{1}{C^{s}(E)}
\end{equation*}%
holds if $B(x,d)\in \mathcal{B}$ as a straightforward consequence of (\ref%
{packing}) and (\ref{centrada}). The arguments given in the proof of (ii)
applied to $\mu $ show that, if $B(x,d)\in \mathcal{B},$ then $\theta _{\mu
}^{s}(x,d)\in Spec(\mu ,\mathcal{O}\cap E),$ which gives the next inclusion
in (iii). The last inclusion follows from the observation that $Spec(\mu ,%
\mathcal{O\cap }E)$ consists of limiting values of sequences with terms in $%
\mathit{Im}(\theta _{\mu }^{s},\mathcal{B}),$ whose extreme values are $%
\frac{1}{P^{s}(E)}$ and $\frac{1}{C^{s}(E)}.$ Using \eqref{muvshaus}, we get
that $\theta _{\alpha }^{s}(x,d)=\alpha (E)\theta _{\mu }^{s}(x,d),$ and
(iii) follows for any $\alpha \in \mathcal{M}^{s}\lfloor _{E}.$
\end{proof}

Of note is the case in which the extreme values of $\theta _{\alpha
}^{s}(x,d)$ are attained on $\mathcal{B}.$ In this case, we have the
following result.

\begin{corollary}
\label{corol} \label{exact}Let $\alpha \in \{\mu ,$ $P^{s}{\lfloor _{E},}$ $%
C^{s}{\lfloor _{E}\}}.$ Under the hypotheses of Theorem \ref{Spectrum}, if
there are two balls $B(x_{1},d_{1})$ and $B(x_{2},d_{2}),$ both in $\mathcal{%
B},$ such that 
\begin{equation}
\theta _{\mu }^{s}(x_{1},d_{1})=\inf \left\{ \theta _{\mu
}^{s}(x,d):B(x,d)\in \mathcal{B}\right\}  \label{inf}
\end{equation}%
and 
\begin{equation}
\theta _{\mu }^{s}(x_{2},d_{2})=\sup \left\{ \theta _{\mu
}^{s}(x,d):B(x,d)\in \mathcal{B}\right\} ,  \label{sup}
\end{equation}%
the inclusions in Theorem~\ref{Spectrum} (iii) can be replaced with
equalities.
\end{corollary}

\begin{proof}
The first inclusion in Theorem \ref{Spectrum}~(iii), together with (\ref%
{packing}), (\ref{centrada}), (\ref{inf}) and (\ref{sup}), implies that 
\begin{equation*}
\mathit{Im}(\theta _{\mu }^{s},\mathcal{B})=\left[ \frac{1}{P^{s}(E)},\frac{1%
}{C^{s}(E)}\right] ,
\end{equation*}%
which, in turn, gives that $\mathit{Im}(\theta _{\alpha }^{s},\mathcal{B}%
)=Spec(\alpha ,\mathcal{O\cap }E).$
\end{proof}

Corollary~\ref{corol} motivates the introduction of the class of $\alpha $-%
\textit{exact self-similar sets} with special properties.

\begin{definition}
\label{Exact}We say that the self-similar set $E$ is $\alpha $-\emph{exact}
if there exists $B\in \mathcal{C}_{\alpha }$ such that 
\begin{equation*}
\frac{\mu (B)}{\left\vert B\right\vert ^{s}}=\sup \left\{ \frac{\mu (B)}{%
\left\vert B\right\vert ^{s}}:B\in \mathcal{C}_{\alpha }\right\}
\end{equation*}%
if $\alpha \in \left\{ C^{s}{\lfloor _{E}},\mathcal{H}^{s}{\lfloor _{E}},%
\mathcal{H}_{Sph}{\lfloor _{E}}\right\} ,$ and 
\begin{equation*}
\frac{\mu (B)}{\left\vert B\right\vert ^{s}}=\inf \left\{ \frac{\mu (B)}{%
\left\vert B\right\vert ^{s}}:B\in \mathcal{C}_{\alpha }\right\} ,
\end{equation*}%
if $\alpha =P^{s}{\lfloor _{E}},$ where $\mathcal{C}_{\alpha }$ is what we
call \textquotedblleft the relevant class of sets\textquotedblright\ for the
measure $\alpha ,$ which is defined as \newline
$\bullet $ $\mathcal{C}_{\alpha }:=$ $\mathcal{B}$ if $\alpha \in \left\{
P^{s}{\lfloor _{E}},C^{s}{\lfloor _{E}}\right\} ,$\newline
$\bullet $ $\mathcal{C}_{\mathcal{H}^{s}{\lfloor _{E}}}:=\{B\subset \mathbb{R%
}^{n}:\ B\ \text{ is a convex set}\}$ and\newline
$\bullet $ $\mathcal{C}_{\mathcal{H}_{Sph}^{s}{\lfloor _{E}}}:=\{B\subset 
\mathbb{R}^{n}:\ B\ \text{ is a closed ball}\}.$
\end{definition}

One nice property that $\alpha $-exact self-similar sets have is that they
possess optimal coverings or packings, that is, almost-coverings (i.e.
coverings for $\alpha $-almost all points in $E)$ or packings whose $s$%
-volume gives the exact value of the corresponding $\alpha $-measure, whilst
if $\alpha $-exactness is not fulfilled, we can only hope to find coverings
or packings with $s$-volume arbitrarily close to the corresponding $\alpha $%
-measure.

\begin{example}
\label{totally disconnected} Self-similar sets $E$ with the strong
separation condition are an example of $\alpha$-exact self-similar sets. See 
\cite{LLM1} for $\alpha\in\left\{ P^{s}{\lfloor_{E}},\text{ }\mathcal{H}^{s}{%
\lfloor_{E}},\text{ }\mathcal{H}_{Sph}^{s}{\lfloor_{E}}\right\} $ and \cite%
{LLM2} for $\alpha=C^{s}.$
\end{example}

\begin{example}
The Sierpinski gasket $S$ is an example of a set where the strong separation
condition does not hold, and that is a $P^{s}{\lfloor }_{S}$-exact (see \cite%
{LLMM2}) and $C^{s}{\lfloor _{S}}$-exact (see \cite{LLMM3}) set.
\end{example}

In \cite{AYER}, it is shown a class of self-similar sets $E$ with OSC in the
line whose members can be non-$\mathcal{H}^{s}{\lfloor _{E}}$-exact (and,
consequently, non-$\mathcal{H}_{Sph}^{s}{\lfloor _{E}}$-exact since these
two measures coincide in the line), and the authors find conditions under
which they are $\mathcal{H}^{s}{\lfloor _{E}}$-exact.

\begin{example}
Self-similar sets $E$ with OSC in the line, with similarity dimension $s,$
and that admit an open interval as a feasible open set, are an example of $%
P^{s}{\lfloor _{E}}$-exact self-similar sets \cite{FENG2}.
\end{example}

\subsection{Complexity of the local structure of self-similar sets\label%
{complexity}}

We now show how these results allow us to explore the complexity of the
local geometric structure of self-similar sets that satisfy the OSC
condition.

First, we need to properly define the equivalence classes of restricted
balls. Notice that different Euclidean balls, even if they share the centre,
can produce the same restricted balls. This motivates the following
definitions that are valid for general subsets of $\mathbb{R}^{n}.$

\begin{definition}
\label{Properdiameter} Given a subset $A\subset \mathbb{R}^{n},$ the
spherical diameter of $A$ is defined by 
\begin{equation*}
\left\vert A\right\vert _{Sph}=\inf \left\{ 2d:A=A\cap B(x,d)\text{ for some 
}x\in A\right\}
\end{equation*}
\end{definition}

\begin{definition}
\label{Restrictedballs} Given a subset $A\subset \mathbb{R}^{n},$ we say
that the restricted ball $B(x,d)\cap A$ is proper and write $B(x,d)\cap A\in 
\mathfrak{P}(A)$ if $x\in A$ and $2d=\left\vert B(x,d)\cap A\right\vert
_{Sph}.$
\end{definition}

\begin{definition}
\label{Spherical density} Given a measure $\alpha$ on $\mathbb{R}^{n}$ and
an $\alpha$-measurable $s$-set $A\subset\mathbb{R}^{n},$ we define the $%
\alpha $-spherical $s$-density of $A$ by 
\begin{equation*}
\theta_{Sph(\alpha)}^{s}(A)=\frac{{\alpha(A)}}{\left( \left\vert
A\right\vert _{Sph}\right) ^{s}}.
\end{equation*}
\end{definition}

\begin{definition}
\label{densityequivalence}Given a subset $A\subset \mathbb{R}^{n}$ and two
restricted balls $B(x,d)\cap A,$ $B(x^{\prime },d^{\prime })\cap A$ both in $%
\mathfrak{P}(A),$ we say that they are similarity-equivalent and write $%
B(x,d)\cap A$ $\simeq _{\mathcal{S}_{n}}B(x^{\prime },d^{\prime })\cap A$ if
there is an $f\in \mathcal{S}_{n}$ such that 
\begin{equation*}
B(x^{\prime },d^{\prime })\cap A=f(B(x,d)\cap A).
\end{equation*}
\end{definition}

\begin{lemma}
\label{equidensityrestricted} Let $A\subset \mathbb{R}^{n}$ and $B(x,d)\cap
A\in \mathfrak{P}(A).$\newline
(i) If $f\in \mathcal{S}_{n}$ has similarity constant $r_{f},$ then $%
f(B(x,d))\cap f(A)\in \mathfrak{P}(f(A))$ and $\left\vert f(B(x,d))\cap
f(A)\right\vert _{p}=r_{f}d.$\newline
(ii) Let $\alpha \in \mathcal{M}^{s}$ and let $A$ be an $\alpha $-measurable 
$s$-set. If $B(x,d)\cap A\simeq _{\mathcal{S}_{n}}B(x^{\prime },d^{\prime
})\cap A,$ then 
\begin{equation*}
\theta _{Sph(\alpha )}^{s}(B(x,d)\cap A)=\theta _{Sph(\alpha
)}^{s}(B(x^{\prime },d^{\prime })\cap A)
\end{equation*}
\end{lemma}

\begin{proof}
Let $A\subset \mathbb{R}^{n},$ $B(x,d)\cap A\in $ $\mathfrak{P}(A).$ In
order to show (i), assume that $f(B(x,d))\cap f(A)$ is not proper. Then,
there is a ball $B(y,\rho )$ such that 
\begin{equation*}
B(y,\rho )\cap f(A)=f(B(x,d))\cap f(A)
\end{equation*}%
with $y\in f(A)$ and $\rho <r_{f}d.$ Then $B(f^{-1}(y),r_{f}^{-1}\rho )\cap
A=B(x,d)\cap A$ with $f^{-1}(y)\in A$ and $r_{f}^{-1}\rho <d,$ in
contradiction with $\left\vert B(x,d)\cap A\right\vert _{Sph}=2d.$
Therefore, $f(B(x,d))\cap f(A)\in \mathfrak{P(}f(A))$ and $\left\vert
f(B(x,d))\cap f(A)\right\vert _{Sph}=2r_{f}d.$\newline
Part (ii) is now trivial since $\alpha \in \mathcal{M}^{s}$ and, hence, 
\begin{equation*}
\alpha (B(x^{\prime },d^{\prime })\cap A)=\alpha (f(B(x,d)\cap
A))=r_{f}^{s}\alpha (B(x,d)\cap A)
\end{equation*}%
and, by (i), 
\begin{equation*}
\left( \left\vert B(x^{\prime },d^{\prime })\cap A\right\vert _{Sph}\right)
^{s}=\left( \left\vert f(B(x,d)\cap f(A))\right\vert _{Sph}\right)
^{s}=r_{f}^{s}(2d)^{s}=r_{f}^{s}\left( \left\vert B(x,d)\cap A\right\vert
_{Sph}\right) ^{s}.
\end{equation*}
\end{proof}

Now we can proceed to state our result for the complexity of the local
geometry of self-similar sets with OSC.

\begin{corollary}
\label{equidensityrestricted copy(1)} Under the assumptions of Theorem~\ref%
{Spectrum}, assume that $s$ is a non-integer real number. Then, there is an
uncountable number of equivalence classes in the quotient space $%
Sph_{E}/\simeq _{\mathcal{S}_{n}}.$
\end{corollary}

\begin{proof}
By Lemma~\ref{equidensityrestricted} {(ii)}, we know that all restricted
balls in an equivalence class of $Sph_{E}/\simeq _{\mathcal{S}_{n}}$ share
the same $\mu $-spherical $s$-density, which allows us to naturally define a
mapping $\theta _{\mu }^{s}:$ $Sph_{E}/\simeq _{\mathcal{S}_{n}}\rightarrow 
\mathit{Im}(\theta _{\mu }^{s},\mathcal{B}).$ This implies that the inverse $%
\left( \theta _{\mu }^{s}\right) ^{-1}:\mathit{Im}(\theta _{\mu }^{s},%
\mathcal{B})\mathbb{\rightarrow }Sph_{E}/\simeq _{\mathcal{S}_{n}}$ of such
mapping is an injective correspondence. Using Marstrand's Theorem, parts
(ii) and (iii) of Theorem~\ref{Spectrum}{\ }and that $\mu (\widehat{E})=1>0,$
it follows that either $C^{s}(E)<P^{s}(E)$ or $s$ is an integer (notice that
from the definitions in Sec. \ref{sectionmetricmeasures} it is easy to see
that $C^{s}(E)\leq P^{s}(E)).$ This, together with Theorem~\ref{Spectrum} {%
(iii)}, means that $\mathit{Im}(\theta _{\mu }^{s},\mathcal{B})$ contains an
interval with uncountably many points and the proof is completed.
\end{proof}

\section{The spectrum of the Sierpinski gasket \label{Sierpinski_gasket}}

In this section, we shall apply the results obtained in Theorem~\ref%
{Spectrum} to fully characterise the asymptotic spectra of the Sierpinski
gasket $S.$

Recall that the Sierpinski gasket or Sierpinski triangle is a special case
of a self-similar set generated by a system $\Psi =\{f_{0,}f_{1,}f_{2}\}$ of
three contracting similitudes of the plane, with contraction ratios $%
r_{i}:=1/2,$ $i\in M,$ given by 
\begin{equation}
f_{0}(x,y)=\frac{1}{2}(x,y),\quad f_{1}(x,y)=\frac{1}{2}(x,y)+(\frac{1}{2}%
,0)\quad \text{and}\quad f_{2}(x,y)=\frac{1}{2}(x,y)+\frac{1}{2}(\frac{1}{2},%
\frac{\sqrt{3}}{2}).  \label{siersyst}
\end{equation}%
We shall denote by $z_{i}$ the fixed point of each $f_{i},$ $i=0,1,2$ that
is, $z_{0}=(0,0),$ $z_{1}=(1,0)$ and $z_{2}=(\frac{1}{2},\frac{\sqrt{3}}{2}%
), $ and by $T$ the equilateral triangle with vertexes $z_{i},$ $\,i\in M.$

It is well known that $S$ is a connected set that satisfies the OSC and has
similarity dimension $s=\frac{\log 3}{\log 2}.$

Thanks to previous work on the packing and Hausdorff centred measures of the
Sierpinski gasket (cf. \cite{LLMM2} and \cite{LLMM3}), we know that $S$ is
both $P^{s}{\lfloor _{S}}$ and $C^{s}{\lfloor _{S}}$-exact, and we have
fairly precise approximations of the values of $P^{s}(S)$ and $C^{s}(S).$

\subsection{Theoretical results}

\begin{theorem}
\label{Sierpinskispectrum} Let $S$ be the Sierpinski gasket, $\widehat{S}%
=\left\{ y\in S:\left\{ \mathcal{T}^{k}(y):k\in \mathbb{N}\right\} \text{ is
dense in }S\right\} ,$ $\mathcal{B}$ be the collection of typical balls, $%
\mathcal{R}$ be a feasible open set for $S,$ and $\alpha \in \mathcal{M}%
^{s}\lfloor _{S}{.}$ Then, the following statements hold true.\newline
(i) 
\begin{equation}
Spec(\alpha ,y)=Spec(\alpha ,\widehat{S})=Spec(\alpha ,\mathcal{R}\cap S)=%
\mathit{Im}(\theta _{\alpha }^{s},\mathcal{B})=\left[ \frac{\alpha (S)}{%
P^{s}(S)},\frac{\alpha (S)}{C^{s}(S)}\right] ,\text{ }y\in \widehat{S}.
\label{Spectrum-sierp}
\end{equation}%
(ii) $Spec(\alpha ,S)$ is given by the union of two closed intervals of
positive length: 
\begin{equation}
Spec(\alpha ,S)=\left[ \underline{\theta }_{\alpha }^{s}(z_{0}),\overline{%
\theta }_{\alpha }^{s}(z_{0})\right] \cup \left[ \frac{\alpha (S)}{P^{s}(S)},%
\frac{\alpha (S)}{C^{s}(S)}\right] ,  \label{sierpinskispectrum}
\end{equation}%
where $z_{0}=(0,0).$ Furthermore, 
\begin{equation}
\underline{\theta }_{\alpha }^{s}(z_{0})=\min \left\{ \theta _{\alpha
}^{s}(z_{0},d):\frac{1}{2}\leq d\leq 1\right\}  \label{minz0}
\end{equation}%
and 
\begin{equation}
\overline{\theta }_{\alpha }^{s}(z_{0})=\max \left\{ \theta _{\alpha
}^{s}(z_{0},d):\frac{1}{2}\leq d\leq 1\right\} .  \label{maxz0}
\end{equation}
\end{theorem}

\begin{proof}
Our previous work guarantees that $S$ is a $P^{s}$-exact (see \cite{LLMM2})
and $C^{s}$-exact (see \cite{LLMM3}) set. Then, the four equalities in (i)
follow as a consequence of Theorem~\ref{Spectrum} and Corollary~\ref{exact}.

In order to prove \eqref{sierpinskispectrum}, let $\mathcal{R}_{i},$ $i\in
\{0,1,2\}$ be the three open rhombi composed of the topological interior of
the union of the triangle $T$ and its reflection across the edge of $T$
opposite the point $z_{i},$ $i\in M$ (see $\mathcal{R}_{2}$ in Fig.~\ref{figura1}). Using that 
\begin{equation}
S=\left\{ z_{0},z_{1},z_{2}\right\} \cup (S\cap \cup _{i=0}^{2}\mathcal{R}%
_{i}),  \label{rhombii}
\end{equation}%
we obtain%
\begin{align*}
Spec(\alpha ,S)& =Spec(\alpha ,S\cap \cup _{i=0}^{2}\mathcal{R}_{i})\cup %
\big(\cup _{i=0}^{2}Spec(\alpha ,z_{i})\big)= \\
& =\left[ \frac{\alpha (S)}{P^{s}(S)},\frac{\alpha (S)}{C^{s}(S)}\right]
\cup Spec(\alpha ,z_{0}),
\end{align*}%
where the last equality follows from ~(\ref{Spectrum-sierp}), (\ref{rhombii}%
) and the fact that, by symmetry, $Spec(\alpha ,z_{i})$ must be identical
for $i\in \left\{ 0,1,2\right\} .$

\begin{figure}[ht]
\centering\includegraphics[scale=0.4]{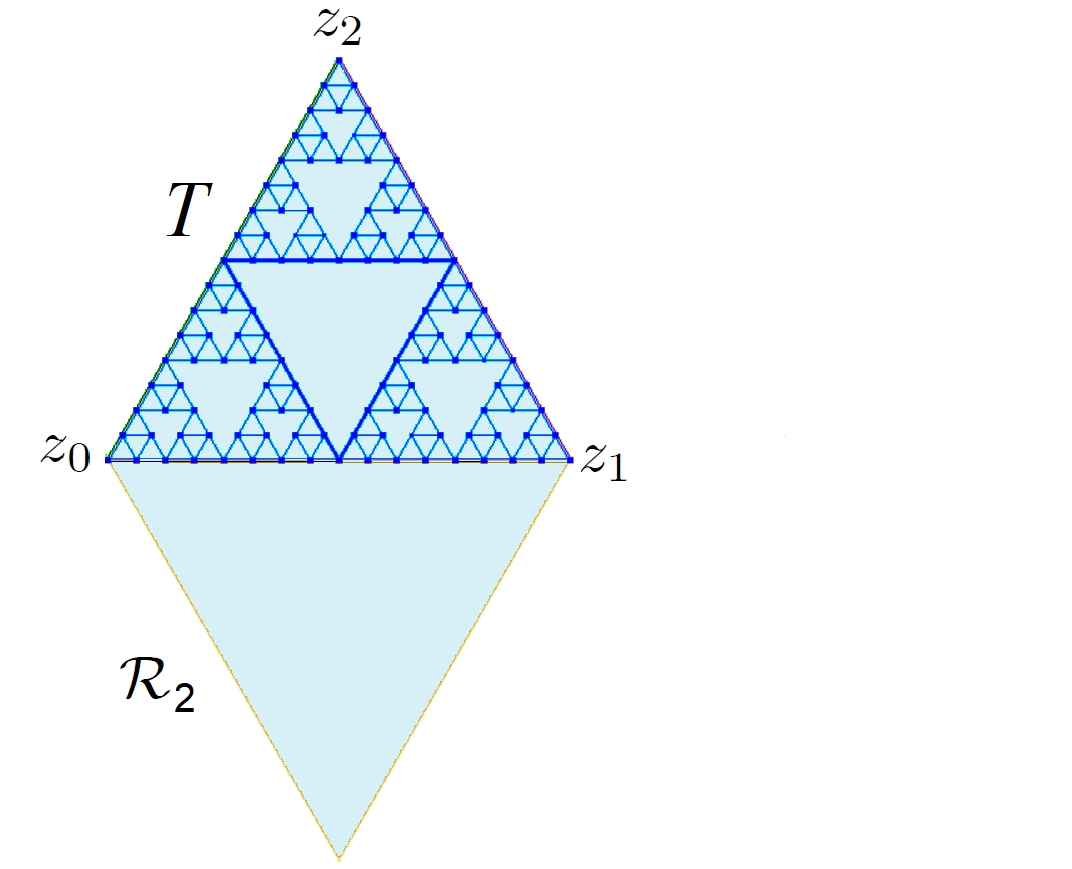}
\caption{A feasible open set. \newline An open rhombus $\mathcal{R}_{2}$ that is a feasible open set for $
S.$}
\label{figura1}
\end{figure}
Observe now that, if $d\leq 1/2,$ then $B(z_{0},d)\cap S=B(z_{0},d)\cap
f_{0}(S).$ Hence, using that $\alpha $ is an $s$-dimensional metric measure 
\begin{align*}
\theta _{\alpha }^{s}(z_{0},d)& =\frac{\alpha (B(z_{0},d)\cap f_{0}(S))}{%
(2d)^{s}}=\frac{\alpha (f_{0}(B(z_{0},2d)\cap S))}{(2d)^{s}} \\
& =\frac{\alpha (B(z_{0},2d)\cap S))}{(4d)^{s}}=\theta _{\alpha
}^{s}(z_{0},2d).
\end{align*}%
If $2d\leq 1/2,$ we can repeat the argument $k$ times until $1/2\leq
2^{k}d\leq 1$ and $\theta _{\alpha }^{s}(z_{0},d)=\theta _{\alpha
}^{s}(z_{0},2^{k}d).$ This shows that 
\begin{equation*}
\min \left\{ \theta _{\alpha }^{s}(z_{0},d):0\leq d\leq 1\right\} =\min
\left\{ \theta _{\alpha }^{s}(z_{0},d):\frac{1}{2}\leq d\leq 1\right\} =%
\underline{\theta }_{\alpha }^{s}(z_{0}),
\end{equation*}%
where the last equality can easily be checked and, analogously, (\ref{maxz0}%
) holds.
\end{proof}

\begin{remark}
\label{exceptional points} Notice that part (i) of Theorem~\ref%
{Sierpinskispectrum} shows that there is a set of full $\alpha $-measure
whose points exhibit a strongly regular behaviour, whereas part {(ii)}
underlines the special local behaviour of the vertexes as the most isolated
points in $S.$ However, the set of exceptional points does not consist only
of the vertexes as there might be other exceptional points, all of them
belonging to the set $\cup _{i=0}^{2}\left( \mathcal{R}_{i}\cap S\right) -%
\widehat{S}.$ The pointwise $\alpha $-density spectrum of such points is
contained in $\left[ \frac{\alpha (S)}{P^{s}(S)},\frac{\alpha (S)}{C^{s}(S)}%
\right] .$ The detection and characterisation of the behaviour of these
points remains an open issue.
\end{remark}

\subsection{Numerical results \label{Sierpinski_gasket-numerical}}

Following the structure of the algorithms developed in \cite{LLM4, LLMM1,
LLMM2, LLMM3} for the numerical estimation of the metric measures of
self-similar sets, the construction of the computational algorithm used in
this work in order to approximate the values of $\underline{\theta }_{\mu
}^{s}(z_{0})$ and $\overline{\theta }_{\mu }^{s}(z_{0})$ relies upon the
discrete approximations of both the Sierpinski gasket and its invariant
measure $\mu .$ Recall that any two measures in $\mathcal{M}^{s}\lfloor _{S}$
\ are mutually multiple of each other (see (\ref{muvshaus})), so we can
obtain $Spec(\alpha ,S)$ from $Spec(\mu ,S)$ if we know $\alpha (S).$

The Sierpinski gasket, as the attractor of $\Psi =\{f_{0},f_{1},f_{2}\}$
(see \eqref{siersyst}), is the unique non-empty compact set that admits the
self-similar decomposition $S=F(S),$ where $F$ is the Hutchinson operator
defined, for $A\subset \mathbb{R}^{2},$ by 
\begin{equation*}
F(A):=f_{0}(A)\cup f_{1}(A)\cup f_{2}(A).
\end{equation*}%
It is well-known that, for any non-empty compact subset $A\subset \mathbb{R}%
^{2},$ $S$ can be built with an arbitrary level of detail by increasing the
iterations $k$ in $F^{k}(A),$ where $F^{k}=F\circ F...\circ F$ is the $k$-th
iterate of the contracting operator $F.$ This is because $F^{k}(A)\overset{%
k\rightarrow \infty }{\rightarrow }S$ in the Hausdorff metric (cf. \cite%
{HUTCH}). Furthermore, if $A\subset S,$ then $F^{k}(A)\subset S$ for any $%
k\in \mathbb{N^{+}}.$ In particular, if we take $A_{1}:=\{z_{0},z_{1},z_{2}%
\} $ as the initial compact set, we obtain the set 
\begin{equation}
A_{k}:=F^{k-1}(A_{1})\subset S,\text{ }k\geq 2,  \label{ak}
\end{equation}%
which approximates $S$ at the iteration $k$ of our algorithm.

The relation between the Markov operator and the natural probability measure 
$\mu _{\mathbf{p}_{s}}$ given in (\ref{markov}), with $s=\frac{\log 3}{\log 2%
}$ and $p_{i}=r_{i}^{s}=3^{-k},$ and (\ref{muvshaus}) leads to the following
relation: 
\begin{equation}
\mathbf{M}_{\mathbf{p}_{s}}^{k}(\alpha )=\frac{1}{3^{k}}\sum_{i\in
M^{k}}\alpha \circ f_{i}^{-1}\overset{w}{\rightarrow }\mu ,\text{ }\quad
\alpha \in \mathcal{P}(\mathbb{R}^{2}).  \label{convergence1}
\end{equation}

If we consider $\mu _{1}:=\frac{1}{3}(\delta _{z_{0}}+\delta _{z_{1}}+\delta
_{z_{2}})$ as an initial measure $\alpha $ in (\ref{convergence1}), where $%
\delta _{x}$ is a unit mass at $x,$ then 
\begin{equation}
\mu _{k}:=\mathbf{M}_{\mathbf{p}_{s}}^{k-1}(\mu _{1})=\frac{1}{3^{k-1}}%
\sum_{i\in M^{k-1}}\mu _{1}\circ f_{i}^{-1}=\frac{1}{3^{k}}\sum_{i\in
M^{k-1}}\left( \delta _{f_{i}(z_{0})}+\delta _{f_{i}(z_{1})}+\delta
_{f_{i}(z_{2})}\right)  \label{mu(k)}
\end{equation}%
is a probability measure supported on $A_{k}\subset S$ and $\mu _{k}\overset{%
w}{\rightarrow }\mu .$ \newline
The discrete measure $\mu _{k}$ is the approximation of the invariant
measure $\mu $ that our algorithm takes at iteration $k.$

Lemmas \ref{union of cylinders} and \ref{lemma cotas} (Lemma \ref{union of
cylinders} is proved in \cite{LLMM3}), provide precise relationships between
the measures $\mu _{k}$ and $\mu .$

\begin{lemma}
\label{union of cylinders} \hfill \break (i) Let $\{S_{i}:i\in I\subset
M^{k}\},$ $k\in\mathbb{N}^{+},$ be a collection of $k$-cylinder sets of $S.$
Then, 
\begin{equation*}
\mu\left( \bigcup\limits_{i\in I}S_{i}\right) \leq\mu_{k}\left(
\bigcup_{i\in I}S_{i}\right)
\end{equation*}
\newline
(ii) Let $A\subset S,$ $k\in\mathbb{N}^{+},$ and let $I=\{i\in
M^{k}:S_{i}\cap A\neq\varnothing\}.$ Then, 
\begin{equation*}
\mu_{k}(A)\leq\mu\left( \bigcup_{i\in I}S_{i}\right)
\end{equation*}

\begin{remark}
\label{open-closed balls}The comparisons between the measures $\mu $ and $%
\mu _{k}$ on collections of cylinders and sets given in the lemma above are
passed to enlarged and reduced balls in part (i) of the next lemma. Since
our algorithms compute only $\mu _{k}$-densities of balls with centres in $%
A_{k}$ (see (\ref{ak})) and with some point of $A_{k}$ in their boundaries,
in part (ii) of this lemma we approximate the $\mu $-measure of a ball
centred at $x$ with the $\mu _{k}$-measure of a ball with its same centre
and with a point of $A_{k}$ at its boundary. \newline
In order to obtain more accurate estimates of $\underline{\theta }_{\mu
}^{s}(z_{0})$ and $\overline{\theta }_{\mu }^{s}(z_{0})$ (as we also do in 
\cite{LLMM2} and \cite{LLMM3} for the estimation of $P^{s}(S)$ and $%
C^{s}(S)),$ it is necessary to consider open balls when searching balls of
minimal $\mu _{k}$-density (see (\ref{densityminz0})), whereas in the search
of balls with maximal $\mu _{k}$-density, the approximating balls must be
taken to be closed balls (see (\ref{densitymaxz0})). In the definition of $%
\underline{\theta }_{\mu }^{s}(\cdot )$ and $\overline{\theta }_{\mu
}^{s}(\cdot ),$ the use of open or closed balls has no relevance because the 
$\mu $-measure of the boundary of any ball is null. However, in the case of
densities of the discrete measures $\mu _{k},$ the values obtained in one or
the other case do actually matter, mainly if $k$ is not large.
\end{remark}
\end{lemma}

From now on, we shall use the notation $\mathring{B}(x,d):=\{y\in \mathbb{R}%
^{2}:|x-y|<d\}$ and $\mathring{\theta}_{\alpha }^{s}$ for the $s$-density of 
$\alpha $ defined using open balls.

\begin{lemma}
\label{lemma cotas}Let $k>0,$ $x\in \mathbb{R}^{2},$ and $2^{-k}<d\leq
\max_{i\in M}\left\Vert z_{i}-x\right\Vert .$ Then,\newline
{(i)} $\mu _{k}(B(x,d-2^{-k}))\leq \mu (B(x,d))\leq \mu _{k}(\mathring{B}%
(x,d+2^{-k}))$\newline
{(ii)} If $B(x,d)\cap A_{k}\neq \varnothing ,$ then there are points $y_{k}$
and $z_{k}$ in $A_{k}$ such that 
\begin{equation*}
\mu _{k}\left( \mathring{B}(x,d_{y_{k}})\right) \leq \mu (B(x,d))\leq \mu
_{k}(B(x,d_{z_{k}})),
\end{equation*}%
where $d_{y_{k}}:=\left\vert y_{k}-x\right\vert ,$ $d_{z_{k}}:=\left\vert
z_{k}-x\right\vert ,$ and $\{d_{y_{k}},d_{z_{k}}\}\in \lbrack
d-2^{-k},d+2^{-k}].$
\end{lemma}

\begin{proof}
\hfill \break (i) Let 
\begin{equation*}
H_{k}:=\{i\in M^{k}:B(x,d-2^{-k})\cap S_{i}\neq \varnothing \}
\end{equation*}%
For any $i\in H_{k},$ $S_{i}\subset B(x,d)$ holds, so $\cup _{i\in
H_{k}}S_{i}\subset B(x,d).$ Using Lemma \ref{union of cylinders} (ii), we
have 
\begin{equation*}
\mu _{k}(B(x,d-2^{-k}))\leq \mu (\cup _{i\in H_{k}}S_{i})\leq \mu (B(x,d)).
\end{equation*}%
Let%
\begin{equation*}
G_{k}:=\{i\in M^{k}:S_{i}\subset \mathring{B}(x,d+2^{-k})\}.
\end{equation*}%
Then, $\mathring{B}(x,d)\cap S\subset \cup _{i\in G_{k}}S_{i}$ and $\cup
_{i\in G_{k}}S_{i}\subset \mathring{B}(x,d+2^{-k}).$ Using Lemma \ref{union
of cylinders} (i), \ we get%
\begin{equation*}
\mu (B(x,d))=\mu (\mathring{B}(x,d)\cap S)\leq \mu (\cup _{i\in
G_{k}}S_{i})\leq \mu _{k}(\cup _{i\in G_{k}}S_{i})\leq \mu _{k}(\mathring{B}%
(x,d+2^{-k}))
\end{equation*}%
(ii) Let $d^{\ast }=\max_{i\in M}\left\Vert z_{i}-x\right\Vert .$ If $%
S\subset B(x,d),$ then $d=d^{\ast }$ and $\mu (B(x,d^{\ast }))=1=\mu
_{k}(B(x,d^{\ast }))>\mu _{k}((\mathring{B}(x,d^{\ast })),$ so property (ii)
holds for $d_{y_{k}}=d_{z_{k}}=d^{\ast }.$ Let us now assume that $%
S\nsubseteq B(x,d).$ We prove first that 
\begin{equation}
F_{k}:=\{i\in M^{k}:\partial B(x,d)\cap S_{i}\neq \varnothing \}\neq
\varnothing .  \label{6}
\end{equation}%
If $F_{k}=\varnothing ,$ then 
\begin{equation*}
\cup _{i\in M^{k}}S_{i}\subset \mathring{B}(x,d)\cup (B(x,d))^{c}.
\end{equation*}%
We know that $(\cup _{i\in M^{k}}S_{i})\cap \mathring{B}(x,d)\neq
\varnothing $ because $B(x,d)\cap A_{k}\neq \varnothing $ and $%
F_{k}=\varnothing ,$ and we also know that $(\cup _{i\in M^{k}}S_{i})\cap
(B(x,d))^{c}\neq \varnothing $ because $S\nsubseteq B(x,d)$ and $%
F_{k}=\varnothing .$ This contradicts that $\cup _{i\in M^{k}}S_{i}$ is a
connected set, and (\ref{6}) must hold.\newline
Using (i), we have that%
\begin{equation*}
\mu (B(x,d))\leq \mu _{k}(B(x,d+2^{-k}))=\mu _{k}(B(x,d_{z_{k}})),
\end{equation*}%
where $z_{k}$ satisfies $d_{z_{k}}=\left\Vert z_{k}-x\right\Vert $ with 
\begin{equation*}
d_{z_{k}}=\max \{\left\Vert y-x\right\Vert :\text{ }y\in A_{k}\cap
B(x,d+2^{-k})\}.
\end{equation*}%
The inequality $d_{z_{k}}\leq d+2^{-k}$ is obvious, and $d_{z_{k}}\geq
d-2^{-k}$ follows because $F_{k}\neq \varnothing $ and each $k$-cylinder $%
S_{i},$ $i\in M^{k}$ contains some point in $A_{k}.$\newline
Using the first inequality in (i), we have%
\begin{equation*}
\mu (B(x,d))\geq \mu _{k}(B(x,d-2^{-k}))=\mu _{k}(\mathring{B}(x,d_{y_{k}})),
\end{equation*}%
where $y_{k}$ satisfies $d_{y_{k}}=\left\Vert y_{k}-x\right\Vert $ with 
\begin{equation*}
d_{y_{k}}=\min \{\left\Vert y-x\right\Vert :\text{ }y\in A_{k}\cap \left(
B(x,d-2^{-k})\right) ^{c}\}.
\end{equation*}%
The inequality $d_{y_{k}}\geq d-2^{-k}$ is obvious, and $d_{y_{k}}\leq
d+2^{-k}$ follows because $F_{k}\neq \varnothing .$
\end{proof}

Theorem \ref{Sierpinskispectrum} allows us to characterise $Spec(\alpha ,S)$
for $\alpha \in \{\mu ,\mathcal{\ }P^{s}\lfloor _{S},$ $C^{s}\lfloor _{S}\}$
through only four numbers, namely, $\underline{\theta }_{\mu }^{s}(z_{0}),%
\overline{\theta }_{\mu }^{s}(z_{0}),$ $P^{s}(S)$ and $C^{s}(S).$ Thanks to
previous numerical work that uses the measures $\mu _{k}$ and the sets $%
A_{k} $ (see (\ref{mu(k)}) and (\ref{ak})) as approximations of $\mu $ and $%
S,$ respectively, we have estimates given by our algorithms $P_{k}$ of $%
P^{s}(S)$ (see \cite{LLMM2}) and $C_{k}$ of $C^{s}(S)$ (see \cite{LLMM3})
and precise error bounds for such estimates. We show in Theorem~\ref{densapp}
below how to obtain estimates $\underline{\xi }_{k}$ of $\underline{\theta }%
_{\mu }^{s}(z_{0})$ and $\overline{\xi }_{k}$ of $\overline{\theta }_{\mu
}^{s}(z_{0}),$ that such estimates converge to the real values, and we give
accurate bounds for them, that is $\underline{\theta }_{\mu }^{s}(z_{0})\in
\lbrack \underline{\xi }_{k}^{\inf },\underline{\xi }_{k}^{\sup }]$ and $%
\overline{\theta }_{\mu }^{s}(z_{0})\in \lbrack \overline{\xi }_{k}^{\inf },%
\overline{\xi }_{k}^{\sup }]$ (see the definition of $\underline{\xi }_{k},$ 
$\overline{\xi }_{k}$ and of the intervals $[\underline{\xi }_{k}^{\inf },%
\underline{\xi }_{k}^{\sup }]$ and $[\overline{\xi }_{k}^{\inf },\overline{%
\xi }_{k}^{\sup }]$ in Theorem~\ref{densapp}). This allows us to implement
an algorithm along the lines of those developed for the estimation of $%
C^{s}(S)$ and $P^{s}(S)$ (see \cite{LLMM2, LLMM3}).

\begin{theorem}
\label{densapp} For $k>1,$ let 
\begin{equation}
\underline{\xi }_{k}:=\min \left\{ \mathring{\theta}_{\mu _{k}}^{s}(z_{0},d):%
\text{ }d=\left\vert x-z_{0}\right\vert ,\text{ }x\in A_{k},\text{ }d\in
\lbrack \frac{1}{2}-2^{-k},1]\right\}  \label{densityminz0}
\end{equation}%
and 
\begin{equation}
\overline{\xi }_{k}:=\max \left\{ \theta _{\mu _{k}}^{s}(z_{0},d):\text{ }%
d=\left\vert x-z_{0}\right\vert ,\text{ }x\in A_{k},\text{ }d\in \lbrack 
\frac{1}{2}-2^{-k},1]\right\}  \label{densitymaxz0}
\end{equation}%
be the estimates of $\underline{\theta }_{\mu }^{s}(z_{0})$ and $\overline{%
\theta }_{\mu }^{s}(z_{0}),$ respectively. Let $d_{k}$ be such that $%
\mathring{\theta}_{\mu _{k}}^{s}(z_{0},d_{k})=\underline{\xi }_{k},$ and let 
$D_{k}$ be such that $\theta _{\mu _{k}}^{s}(z_{0},D_{k})=\overline{\xi }%
_{k}.$ \newline
Then, 
\begin{equation}
\{\underline{\theta }_{\mu }^{s}(z_{0}),\underline{\xi }_{k}\}\in \lbrack 
\underline{\xi }_{k}^{\inf },\underline{\xi }_{k}^{\sup }],
\label{thetainfbounds}
\end{equation}%
and 
\begin{equation}
\{\overline{\theta }_{\mu }^{s}(z_{0}),\overline{\xi }_{k}\}\in \lbrack 
\overline{\xi }_{k}^{\inf },\overline{\xi }_{k}^{\sup }],
\label{thetasupbounds}
\end{equation}%
where 
\begin{equation}
\underline{\xi }_{k}^{\inf }=\underline{K}_{k}\underline{\xi }_{k},\quad 
\underline{K}_{k}=(1-2^{1-k})^{s},\quad \underline{\xi }_{k}^{\sup }=\frac{%
\mu _{k}(\mathring{B}(z_{0},d_{k}+2^{-k}))}{(2d_{k})^{s}},
\label{theta bounds and k}
\end{equation}%
\begin{equation}
\overline{\xi }_{k}^{\inf }=\frac{\mu _{k}(B(z_{0},D_{k}-2^{-k}))}{%
(2D_{k})^{s}},\quad \overline{K}_{k}=(1+2^{1-k})^{s},\quad \overline{\xi }%
_{k}^{\sup }=\overline{K}_{k}\overline{\xi }_{k}.
\label{thetasup bounds and k}
\end{equation}
\end{theorem}

\begin{proof}
That $\underline{\xi }_{k}\in \lbrack \underline{\xi }_{k}^{\inf },%
\underline{\xi }_{k}^{\sup }]$ and $\overline{\xi }_{k}\in \lbrack \overline{%
\xi }_{k}^{\inf },\overline{\xi }_{k}^{\sup }]$ is obvious from the
definitions.\newline
We prove first that $\underline{\theta }_{\mu }^{s}(z_{0})\in \lbrack 
\underline{\xi }_{k}^{\inf },\underline{\xi }_{k}^{\sup }].$ Using Lemma \ref%
{lemma cotas}~(i) and (\ref{minz0}), we obtain 
\begin{equation*}
\underline{\theta }_{\mu }^{s}(z_{0})\leq \frac{\mu (B(z_{0},d_{k}))}{%
(2d_{k})^{s}}\leq \frac{\mu _{k}(\mathring{B}(z_{0},d_{k}+2^{-k}))}{%
(2d_{k})^{s}}=\underline{\xi }_{k}^{\sup }.
\end{equation*}%
Let $d\in \lbrack \frac{1}{2},1]$ be such that $\underline{\theta }_{\mu
}^{s}(z_{0})=\frac{\mu (B(z_{0},d))}{(2d)^{s}}.$ Lemma \ref{lemma cotas}%
~(ii) guarantees the existence of $y_{k}\in A_{k}$ such that $\mu
(B(z_{0},d))\geq \mu _{k}(\mathring{B}(z_{0},d_{y_{k}})),$ where $%
d_{y_{k}}:=\left\vert y_{k}-z_{0}\right\vert \in \lbrack
d-2^{-k},d+2^{-k}]\subset \lbrack \frac{1}{2}-2^{-k},1].$ This, together
with (\ref{densityminz0}) and the inequality $d\geq \frac{1}{2}$ gives 
\begin{align*}
\underline{\theta }_{\mu }^{s}(z_{0})& =\frac{\mu (B(z_{0},d))}{(2d)^{s}}%
\geq \frac{\mu _{k}(\mathring{B}(z_{0},d_{y_{k}}))}{(2d)^{s}}=\left( \frac{%
d_{y_{k}}}{d}\right) ^{s}\frac{\mu _{k}(\mathring{B}(z_{0},d_{y_{k}}))}{%
(2d_{y_{k}})^{s}} \\
& \geq \left( \frac{d_{y_{k}}}{d}\right) ^{s}\underline{\xi }_{k}\geq \left( 
\frac{d-2^{-k}}{d}\right) ^{s}\underline{\xi }_{k}\geq \underline{\xi }%
_{k}^{\inf }.
\end{align*}%
The proof that $\overline{\theta }_{\mu }^{s}(z_{0})\in \lbrack \overline{%
\xi }_{k}^{\inf },\overline{\xi }_{k}^{\sup }]$ is analogous. Using Lemma %
\ref{lemma cotas}~(i) and (\ref{maxz0}), we obtain%
\begin{equation*}
\overline{\theta }_{\mu }^{s}(z_{0})\geq \frac{\mu (B(z_{0},D_{k}))}{%
(2D_{k})^{s}}\geq \frac{\mu _{k}(B(z_{0},D_{k}-2^{-k}))}{(2D_{k})^{s}}=%
\overline{\xi }_{k}^{\inf }.
\end{equation*}%
Let $D\in \lbrack \frac{1}{2},1]$ be such that $\overline{\theta }_{\mu
}^{s}(z_{0})=\frac{\mu (B(z_{0},D))}{(2D)^{s}}.$ Lemma \ref{lemma cotas}%
~(ii) guarantees the existence of $z_{k}\in A_{k}$ such that $\mu
(B(z_{0},D))\leq \mu _{k}(B(z_{0},d_{z_{k}})),$ where $d_{z_{k}}:=\left\vert
z_{k}-z_{0}\right\vert \in \lbrack D-2^{-k},D+2^{-k}]\subset \lbrack \frac{1%
}{2}-2^{-k},1].$ This, together with (\ref{densitymaxz0}) and the inequality 
$D\geq \frac{1}{2}$ gives 
\begin{align*}
\overline{\theta }_{\mu }^{s}(z_{0})& =\frac{\mu (B(z_{0},D))}{(2D)^{s}}\leq 
\frac{\mu _{k}(B(z_{0},d_{z_{k}}))}{(2D)^{s}} \\
& =\left( \frac{d_{z_{k}}}{D}\right) ^{s}\frac{\mu _{k}(B(z_{0},d_{z_{k}}))}{%
(2d_{z_{k}})^{s}}\leq \left( \frac{d_{z_{k}}}{D}\right) ^{s}\overline{\xi }%
_{k} \\
& \leq \left( \frac{D+2^{-k}}{D}\right) ^{s}\overline{\xi }_{k}\leq 
\overline{\xi }_{k}^{\sup }.
\end{align*}
\end{proof}

We present in Table \ref{numresults1} the estimates $\underline{\xi }_{k},$
and $\overline{\xi }_{k}$ of $\underline{\theta }_{\mu }^{s}(z_{0})$ and $%
\overline{\theta }_{\mu }^{s}(z_{0})$ (see (\ref{thetainfbounds}) and (\ref%
{thetasupbounds}) for definitions)$,$ respectively, and the corresponding
lower and upper bounds in the 100\% confidence intervals $[\underline{\xi }%
_{k}^{\inf },\underline{\xi }_{k}^{\sup }],$ $[\overline{\xi }_{k}^{\inf },%
\overline{\xi }_{k}^{\sup }]$ (see (\ref{thetasup bounds and k}),(\ref%
{thetasupbounds})) obtained by our algorithm for $k=14$ (see the definition
these values in (\ref{densityminz0}), (\ref{densitymaxz0}), (\ref{theta
bounds and k}) and (\ref{thetasup bounds and k})). We also provide the
radii, $d_{k}$ and $D_{k},$ of the $\mu _{k}$-optimal balls.

See in Fig.~\ref{densgraphs1} the graph of the function $\theta _{\mu
_{14}}^{s}(z_{0},d)$ as a function of $d\in \lbrack \varepsilon ,1],$ and in
Fig.~\ref{densgraphs2} the points $(g(d),\theta _{\mu _{14}}^{s}(z_{0},d)),$
where $g(d):=\varepsilon +\frac{\varepsilon -1}{\log (\varepsilon )}(\log
(d)-\log (\varepsilon ))$ and $\varepsilon =0.05.$ This is a suitable
logarithmic scale, \cite{BAN}, which allows us to see the periodicity of
this function at such a scale.

\begin{figure}[H]
\centering%
\begin{subfigure}[b]{0.5\textwidth}
\centering
\includegraphics[width=\textwidth]{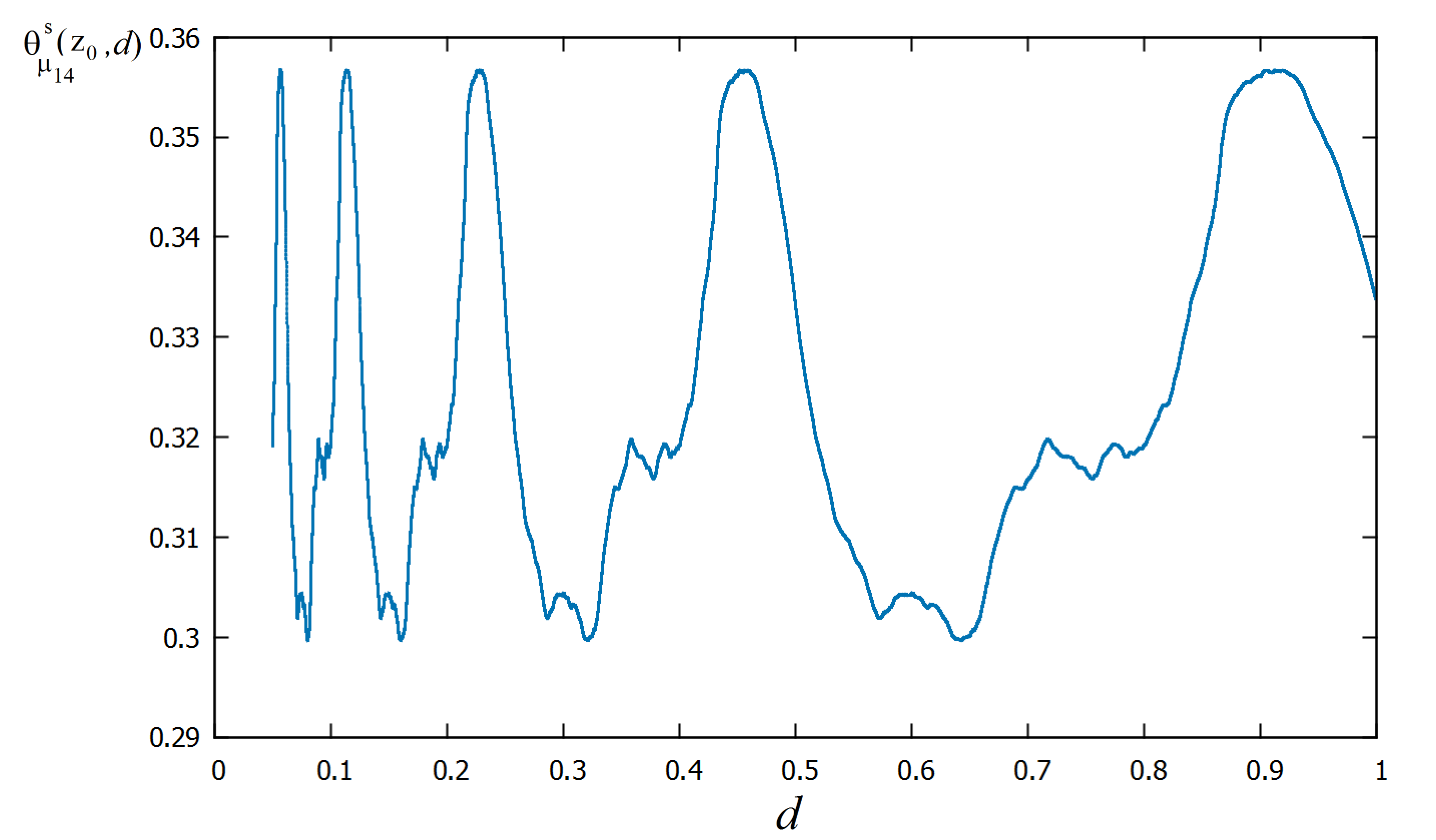}
\caption{Values of $\theta _{\mu _{14}}^{s}(z_{0},d)$ for $d\in \lbrack \varepsilon ,1]$ and $\varepsilon =0.05.$}
\label{densgraphs1}
\end{subfigure}\hfill 
\begin{subfigure}[b]{0.47\textwidth}
\centering
\includegraphics[width=\textwidth]{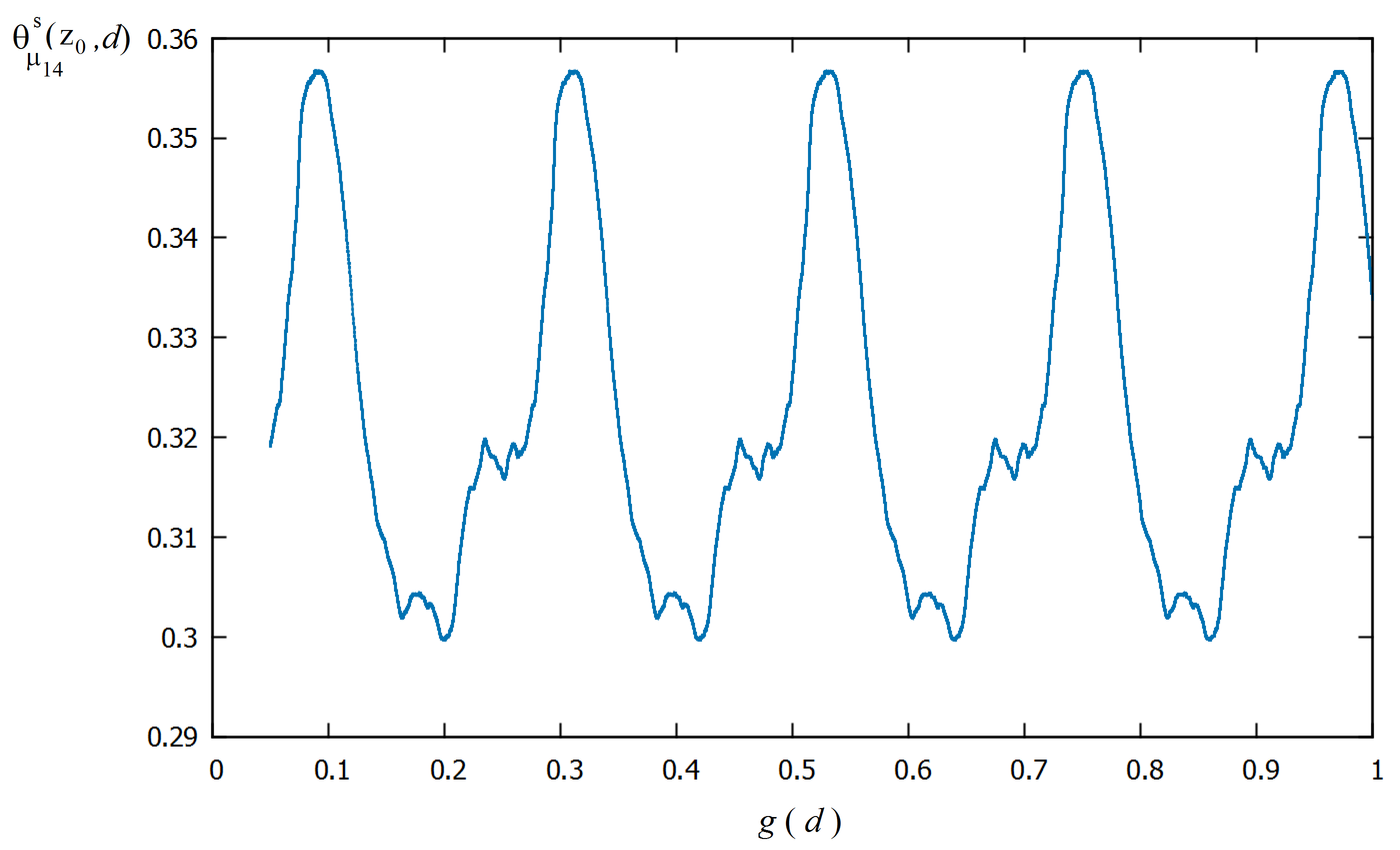}
\caption{Values of $(g(d),\theta _{\mu _{14}}^{s}(z_{0},d)),$ where $g(d):=\varepsilon +\frac{\varepsilon -1}{\log (\varepsilon )}(\log (d)-\log (\varepsilon ))$ and $\varepsilon =0.05.$}
\label{densgraphs2}
\end{subfigure}
\caption{ Densities at $z_{0}$ }
\end{figure}

We present in Table \ref{numresults2} the estimates $P_{k}$ of $P^{s}(S)$
and $C_{k}$ of $C^{s}(S)$ obtained by our algorithms for $k=14.$ The lower
and upper bounds of $P^{s}(S)$ are denoted by $P_{k}^{\inf }$ and $%
P_{k}^{\sup },$ respectively, and the bounds of $C^{s}(S)$ are denoted by $%
C_{k}^{\inf }$ and $C_{k}^{\sup },$ respectively. These results were
computed in \cite{LLMM2} and \cite{LLMM3}), respectively. Recall that $%
\left( P^{s}(S)\right) ^{-1}$ and $\left( C^{s}(S)\right) ^{-1}$ are the $%
\mu $-densities of the balls of minimum and maximum $\mu $-density in the
set of typical balls. The estimates $P_{k}$ and $C_{k}$ are obtained by
replacing $S$ with $A_{k}$ and $\mu $ with $\mu _{k}.$ Again, we have used
open balls in the estimation of the density of the ball of minimum $\mu _{k}$%
-density, and closed balls for the density of the ball of maximum $\mu _{k}$%
-density. The centre and radius of the open ball of minimum $\mu _{k}$%
-density are denoted by $x_{k}^{\ast }$ and $d_{k},$ respectively, and the
centre and radius corresponding to the closed ball of maximum $\mu _{k}$%
-density are denoted by $y_{k}^{\ast }$ and $D_{k}.$ The table also contains
the corresponding optimal $\mu _{k}$-densities and their bounds. The upper
bound $P_{k}^{\sup }:=K_{k}^{P}P_{k}$ \ of $P^{s}(S)$ is slightly improved
here with respect to the one given in \cite{LLMM2}. Here $K_{k}^{P}:=(1-%
\frac{2^{5-k}}{\sqrt{3}})^{-s}$ instead of the value $K_{k}^{P}=(1-\frac{%
2^{6-k}}{\sqrt{3}})^{-s}$ used in \cite{LLMM2}. This gives the value $%
P_{14}^{\sup }=1.671292$ given in Table \ref{numresults2} instead of the
value $P_{14}^{\sup }=1.668305$ given in Table 1 in \cite{LLMM2}. \newline
The results of the following corollary are based on the estimates of Tables %
\ref{numresults1} and \ref{numresults2}.

\begin{corollary}
Let $S$ be the Sierpinski gasket. \newline
(i) For any $\alpha\in \mathcal{M}^{s}\lfloor_{S},$ $Spec(\alpha,S)$ is the
union of two closed disjointed intervals.\newline
(ii)%
\begin{align*}
Spec(\mu,S) & \sim\left[ 0.2997,0.3567\right] \cup\left[ 0.5994,0.9951\right]
\\
\left[ 0.2998,0.3566\right] \cup\left[ 0.5999,0.9944\right] & \subset
Spec(\mu,S)\subset\left[ 0.2996,0.3568\right] \cup\left[ 0.5983,0.9970\right]
\end{align*}
(iii)%
\begin{align*}
Spec(P^{s}\lfloor_{S},S) & \sim\left[ 0.5,0.5951\right] \cup\left[ 1,1.6602%
\right] \\
\lbrack0.5010,0.5945]\cup\lbrack1,1.6578] & \subset Spec(P^{s}\lfloor
_{S},S)\subset\lbrack0.4995,0.5963]\cup\lbrack1,1.6662]
\end{align*}
(iv)%
\begin{align*}
Spec(C^{s}\lfloor_{S},S) & \sim\left[ 0.3012,0.3584\right] \cup\left[
0.6023,1\right] \\
\lbrack0.3015,0.3577]\cup\lbrack0.6032,1] & \subset Spec(C^{s}\lfloor
_{S},S)\subset\lbrack0.3005,0.3588]\cup\lbrack0.6002,1]
\end{align*}
\end{corollary}

\begin{proof}
We know (see (\ref{sierpinskispectrum}) in Theorem \ref{Spectrum}) that 
\begin{equation}
Spec(\mu ,S)=\left[ \underline{\theta }_{\mu }^{s}(z_{0}),\overline{\theta }%
_{\mu }^{s}(z_{0})\right] \cup \left[ \frac{1}{P^{s}(S)},\frac{1}{C^{s}(S)}%
\right] ,  \label{intervals}
\end{equation}%
and that (see (\ref{muvshaus})) 
\begin{equation}
Spec(\alpha ,S)=\alpha (S)Spec(\mu ,S),\text{ \ }\alpha \in \mathcal{M}%
^{s}\lfloor _{S}.  \label{spectrum-equiv}
\end{equation}%
The two intervals in $Spec(\alpha ,S),$ $\alpha \in \mathcal{M}^{s}\lfloor
_{S}$ are disjointed if $\overline{\theta }_{\mu }^{s}(z_{0})<\frac{1}{%
P^{s}(S)}.$ Such a condition holds (see Theorem \ref{densapp}, and Tables %
\ref{numresults1} and \ref{numresults2}) because 
\begin{equation*}
\overline{\theta }_{\mu }^{s}(z_{0})\leq \overline{\xi }_{14}^{\sup }<0.3568
\end{equation*}%
and 
\begin{equation*}
\frac{1}{P^{s}(S)}\geq \frac{1}{P_{14}^{\sup }}>0.5983.
\end{equation*}%
Using (\ref{intervals}), Theorem \ref{densapp} and (\ref{spectrum-equiv}),
we have that 
\begin{align*}
Spec(\mu ,S)& \sim \lbrack \underline{\xi }_{14},\overline{\xi }_{14}]\cup %
\left[ \frac{1}{P_{14}},\frac{1}{C_{14}}\right] , \\
\left[ \underline{\xi }_{14}^{\sup },\overline{\xi }_{14}^{\inf }\right]
\cup \left[ \frac{1}{P_{14}^{\inf }},\frac{1}{C_{14}^{\sup }}\right] &
\subset Spec(\mu ,S)\subset \left[ \underline{\xi }_{14}^{\inf },\overline{%
\xi }_{14}^{\sup }\right] \cup \left[ \frac{1}{P_{14}^{\sup }},\frac{1}{%
C_{14}^{\inf }}\right] ,
\end{align*}%
\begin{align*}
Spec(P^{s}\lfloor _{S},S)& \sim \lbrack P_{14}\underline{\xi }_{\mu
_{14}},P_{14}\overline{\xi }_{\mu _{14}}]\cup \left[ 1,\frac{P_{14}}{C_{14}}%
\right] , \\
\left[ P_{14}^{\sup }\underline{\xi }_{14}^{\sup },P_{14}^{\inf }\overline{%
\xi }_{14}^{\inf }\right] \cup \left[ 1,\frac{P_{14}^{\inf }}{C_{14}^{\sup }}%
\right] & \subset Spec(P^{s}\lfloor _{S},S)\subset \left[ P_{14}^{\inf }%
\underline{\xi }_{14}^{\inf },P_{14}^{\sup }\overline{\xi }_{14}^{\sup }%
\right] \cup \left[ 1,\frac{P_{14}^{\sup }}{C_{14}^{\inf }}\right] ,
\end{align*}%
(and, analogously, for $Spec(C^{s}\lfloor _{S},S)),$ and the proof is
completed using the corresponding estimates of Tables \ref{numresults1} and %
\ref{numresults2}.
\begin{table}[H]
\centering$%
\begin{tabular}{ccc}
\hline
\multicolumn{1}{|c}{$\underline{\xi}_{14}$} & \multicolumn{1}{|c}{$[%
\underline {\xi}_{14}^{\inf},\underline{\xi}_{14}^{\sup}]$} & 
\multicolumn{1}{|c|}{$d_{14}$} \\ \hline
\multicolumn{1}{|l}{0.299714} & \multicolumn{1}{|l}{[0.299656,0.299763]} & 
\multicolumn{1}{|l|}{0.642272} \\ \hline
&  &  \\ \hline
\multicolumn{1}{|c}{$\overline{\xi}_{14}$} & \multicolumn{1}{|c}{$[%
\overline {\xi}_{14}^{\inf},\overline{\xi}_{14}^{\sup}]$} & 
\multicolumn{1}{|c|}{$D_{14}$} \\ \hline
\multicolumn{1}{|l}{0.356687} & \multicolumn{1}{|l}{[0.356645,0.356756]} & 
\multicolumn{1}{|l|}{0.913663} \\ \hline
\end{tabular}
\ \ $%
\caption{ Extreme densities at $z_{0}$\newline Estimates of $\protect\underline{\protect\theta }_{\protect\mu %
}^{s}(z_{0})$ and $\overline{\protect\theta }_{\protect\mu }^{s}(z_{0}),$
bounds and radii, $d_{k}$ and $D_{k},$ of the $\protect\mu _{k}$-optimal
balls for $k=14.$}
\label{numresults1}
\end{table}
\begin{table}[H]
\centering$%
\begin{tabular}{|c|c|c|c|c|c|}
\hline\hline
$x_{14}^{\ast }$ & $d_{14}$ & $P_{14}$ & $[P_{14}^{\inf },P_{14}^{\sup }]$ & 
$\left( P_{14}\right) ^{-1}=\mathring{\theta}_{\mu _{14}}^{s}(x_{14}^{\ast
},d_{14})$ & $\ \left[ \left( P_{14}^{\sup }\right) ^{-1},\left(
P_{14}^{\inf }\right) ^{-1}\right] $ \\ \hline
(0.5,0) & \multicolumn{1}{|l|}{0.160543} & \multicolumn{1}{|l|}{1.668305} & 
\multicolumn{1}{|l|}{[1.667178, 1.671292]} & 0.599411 & [0.598339, 0.599816]
\\ \hline\hline
$y_{14}^{\ast }$ & $D_{14}$ & $C_{14}$ & $[C_{14}^{\inf },C_{14}^{\sup }]$ & 
$\left( C_{14}\right) ^{-1}=\theta _{\mu _{14}}^{s}(y_{14}^{\ast },D_{14})$
& $\left[ \left( C_{14}^{\sup }\right) ^{-1},\left( C_{14}^{\inf }\right)
^{-1}\right] $ \\ \hline
$\left( \frac{5}{16},\frac{\sqrt{3}}{16}\right) $ & \multicolumn{1}{|l|}{
0.145957} & \multicolumn{1}{|l|}{1.004903} & \multicolumn{1}{|l|}{
[1.003109,1.005611]} & 0.995121 & [0.994420, 0.996901] \\ \hline
\end{tabular}%
\ \ \ $%
\caption{Packing and Centred measure estimates of $S$\newline Centres and radii of the balls $\mathring{B}(x_{14}^{\ast },d_{14})$
and $B(y_{14}^{\ast },D_{14})$ of minimum and maximum $\protect\mu _{14}$%
-densities, estimates $P_{14}$ and $C_{14}$ of $P^{s}(S)$ and $C^{s}(S),$
and bounds. The last two columns in the table are the $\protect\mu _{14}$%
-densities of the optimal balls (inverses of $P^{s}(S)$ and $C^{s}(S))$ and
their bounds.}
\label{numresults2}
\end{table}
\end{proof}

\begin{acknowledgement}
This work was supported by the Universidad Conmplutense de Madrid and the
Banco de Santander (PR108/20-14).
\end{acknowledgement}

\end{document}